\documentclass[11pt]{amsart}

\usepackage{amsfonts,amsthm,amsmath,enumerate,amssymb,latexsym,color,tcolorbox,tikz,mathrsfs,bm}
\usepackage{subfig}
\usepackage[text={6in,9in},centering]{geometry}
\usepackage[colorlinks,
            linkcolor=blue,
            anchorcolor=red,
            citecolor=red]{hyperref}
\usetikzlibrary{arrows,shapes,bending}

\graphicspath{{figures/}}
\usepackage{epstopdf}

\theoremstyle{plain}
\newtheorem{theorem}{Theorem}[section]
\newtheorem{lemma}[theorem]{Lemma}
\newtheorem{corollary}[theorem]{Corollary}
\newtheorem{proposition}[theorem]{Proposition}

\theoremstyle{definition}
\newtheorem{definition}[theorem]{Definition}
\newtheorem{example}[theorem]{Example}

\theoremstyle{remark}
\newtheorem{remark}[theorem]{Remark}

\numberwithin{equation}{section}

\DeclareMathOperator{\id}{id}
\DeclareMathOperator{\ter}{ter}

\def\R{\mathbb R}
\def\Z{\mathbb Z}
\def\D{\mathcal D}

\def\V{\mathcal V}
\def\E{\mathcal E}
\def\I{\mathcal I}

\def\vp{\varphi}

\begin{document}

\title[Connectedness of Sierpi\'nski sponges and associated graph-directed systems]{The connectedness of Sierpi\'nski sponges with rotational and reflectional components and associated graph-directed systems}

\author{Huo-Jun Ruan}
\address{School of Mathematical Sciences, Zhejiang University, Hangzhou, China}
\email{ruanhj@zju.edu.cn}

\author{Jian-Ci Xiao}
\address{School of mathematics, Nanjing University of Aeronautics and Astronautics, Nanjing, China}
\email{jcxiao@nuaa.edu.cn}



\thanks{Corresponding author: Jian-Ci Xiao}

\begin{abstract}
    We provide two methods to characterize the connectedness of all $d$-dimensional generalized Sierpi\'nski sponges whose corresponding IFSs are allowed to have rotational and reflectional components. Our approach is to reduce it to an intersection problem between the coordinates of graph-directed attractors. More precisely, let $(K_1,\ldots,K_n)$ be a Cantor-type graph-directed attractor in $\R^d$. By creating an auxiliary graph, we provide an effective criterion for whether $K_i\cap K_j$ is empty for every pair of $1\leq i,j\leq n$. Moreover, the emptiness can be checked by examining only a finite number of geometric approximations of the attractor. The approach is also applicable to more general graph-directed systems.
\end{abstract}

\maketitle
%\tableofcontents

%---------------------
\section{Introduction}

Connectedness is a fundamental aspect of the topology of fractal sets. However, to tell whether a given fractal is connected or not is usually a nontrivial question. Existing results focus on special classes of self-similar or self-affine sets, see~\cite{DLRWX21,DL11,SS06,Sol15,SS20}. In a classical paper~\cite{Hat85}, Hata elegantly transformed the connectedness problem of attractors of IFSs into a connectedness problem of graphs.

\begin{theorem}[Hata~\cite{Hat85}]
    Let $\Phi=\{\varphi_i\}_{i=1}^n$ be an IFS on $\R^d$ and let $K$ be its attractor. Then $K$ is connected if and only if the associated Hata graph is connected, where the \emph{Hata graph} is defined as follows.
    \begin{enumerate}
        \item The vertex set is the index set $\{1,\ldots,n\}$;
        \item There is an edge joining distinct $1\leq i,j\leq n$ if and only if $\varphi_i(K)\cap\varphi_j(K)\neq\varnothing$.
    \end{enumerate}
\end{theorem}

As simple as it seems, the examination of whether $\varphi_i(K)\cap\varphi_j(K)\neq\varnothing$ is a hard task in many circumstances, even when $K$ is self-similar. Since all the ingredients in hand are the IFSs, we usually have to choose a suitable compact invariant set (with respect to the IFS), iterate it several times, and hope to obtain useful information on the connectedness of the limit set. In this paper, we shall focus on ``sponge-like'' self-similar sets defined as follows. Denote $\mathcal{O}_d$ to be the group of symmetries of the $d$-cube $[0,1]^d$. That is to say, $\mathcal{O}_d$ consists of all isometries that map $[0,1]^d$ onto itself, including both orientation preserving and orientation reversing ones. Let $N\geq 2$ be an integer and let $\I\subset\{0,1,\ldots,N-1\}^d$ be a non-empty set with $1<|\I|<N^d$, where $|\I|$ denotes the cardinality of $\I$. Suppose for each $i\in\I$ there corresponds a contracting map
\[
    \varphi_i(x) = \frac{1}{N}(O_i(x)+i), \quad x\in\R^d,
\]
where $O_i\in\mathcal{O}_d$. A classical result of Hutchinson~\cite{Hut81} states that there is a unique non-empty compact set $F=F(d,N,\I)$ such that $F=\bigcup_{i\in\I}\vp_i(F)$. When $O_i=\id$ for all $i\in\I$, the attractor $F$ is usually called a \emph{Sierpi\'nski sponge} when $d\geq 3$ and a \emph{generalized Sierpi\'nski carpet} or a \emph{fractal square} when $d=2$. For convenience, we shall name the attractor $F$ (in our setting) a \emph{sponge-like set} throughout this paper. These sets can be obtained by the following geometric iteration process: Writing $F_0:=[0,1]^d$ and recursively defining 
\begin{equation}\label{eq:notationoffk}
    F_k:=\bigcup_{i\in\I} \vp_i(F_{k-1}), \quad k\geq 1,
\end{equation}
then $F=\bigcap_{k=0}^\infty F_k$. Note that $F_k$ is a finite union of closed cubes of side length $N^{-k}$.

% In this section, we denote $\mathcal{O}$ to be the group of symmetries of the unit square $[0,1]^2$. It is well known that the group $\mathcal{O}$ consists of $8$ elements, including
% \begin{enumerate}
%     \item Rotations of $0^\circ$ (i.e., the identity map), $90^\circ$, $180^\circ$ and $270^\circ$;
%     \item Reflections (or flips) along vertical, horizontal, and the two diagonal axes.
% \end{enumerate}

\begin{example}\label{exa:carpet1}
    Let $d=2$, $N=2$ and let $\I=\{(0,0),(1,0),(0,1)\}$. In Figure~\ref{fig:carpet1}, we show attractors associated with the following three IFSs from left to right:
    \begin{enumerate}
        \item $O_i\equiv\id$ for $i\in\I$;
        \item $O_{(0,0)}=\id$, while $O_{(0,1)}, O_{(1,0)}$ are rotations of $90^\circ$ and $270^\circ$ (clockwise), respectively;
        \item $O_{(0,1)}$ is the rotation of $90^\circ$ (counterclockwise), while $O_{(0,0)},O_{(1,0)}$ are flips along the lines $x=\frac{1}{2}$ and $y=\frac{1}{2}$, respectively.
    \end{enumerate}
    Interested readers can find an illustration of all possible attractors in the case when $d=2$, $N=2$ and $|\I|=3$ in the book~\cite[Section 5.3]{PJS04}.
\end{example}

\begin{figure}[htbp]
    \centering
    \includegraphics[width=3.3cm]{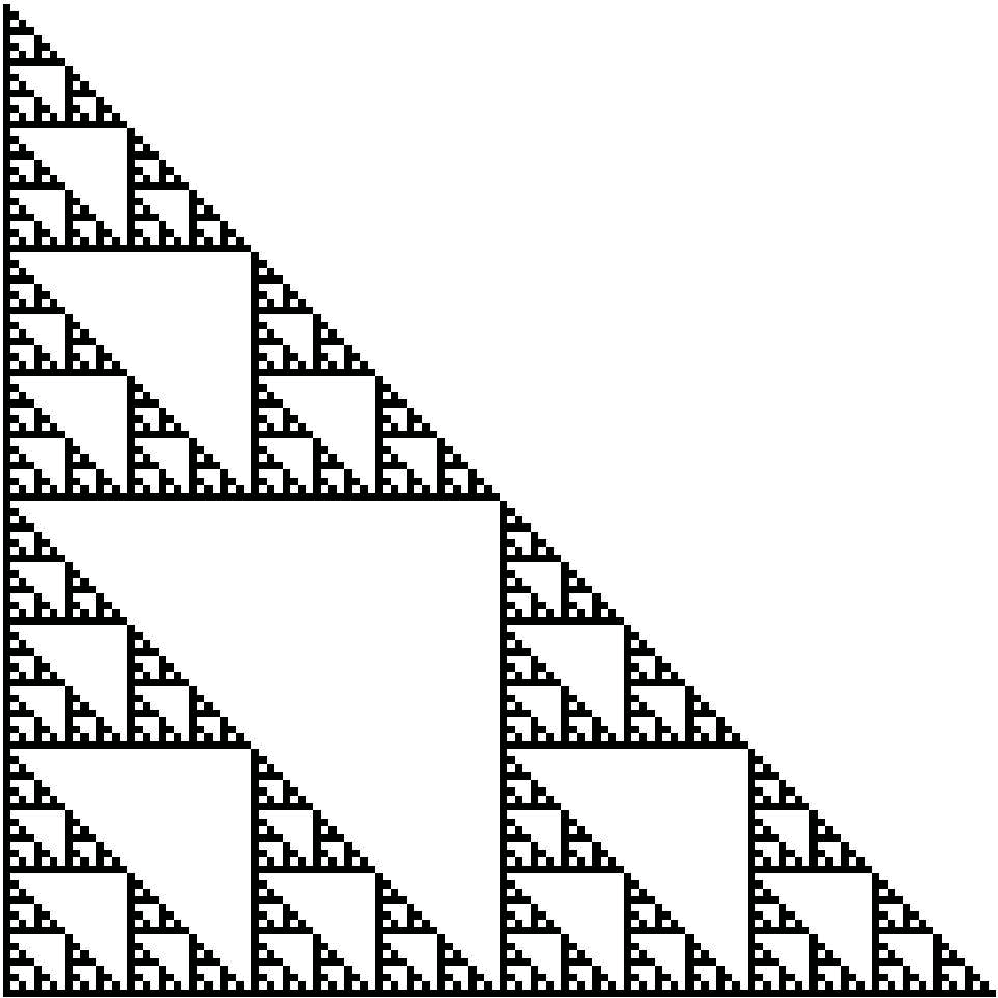}\quad\,
    \includegraphics[width=3.3cm]{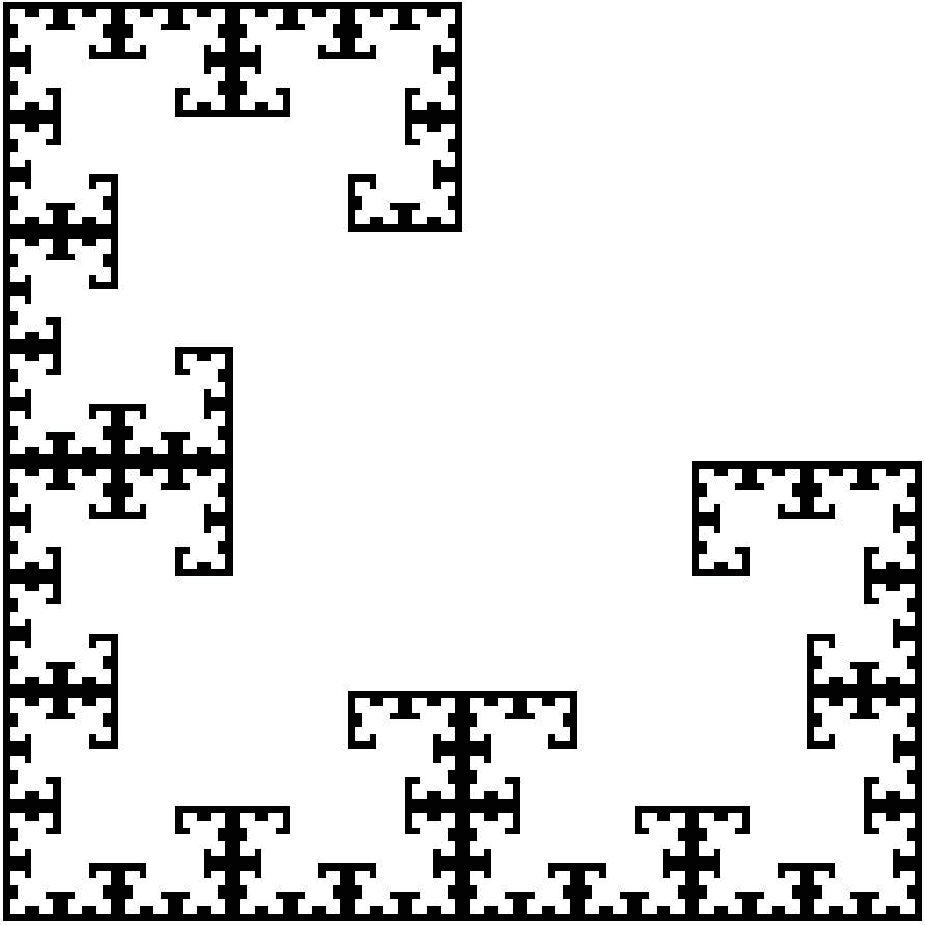}\quad\,
    \includegraphics[width=3.3cm]{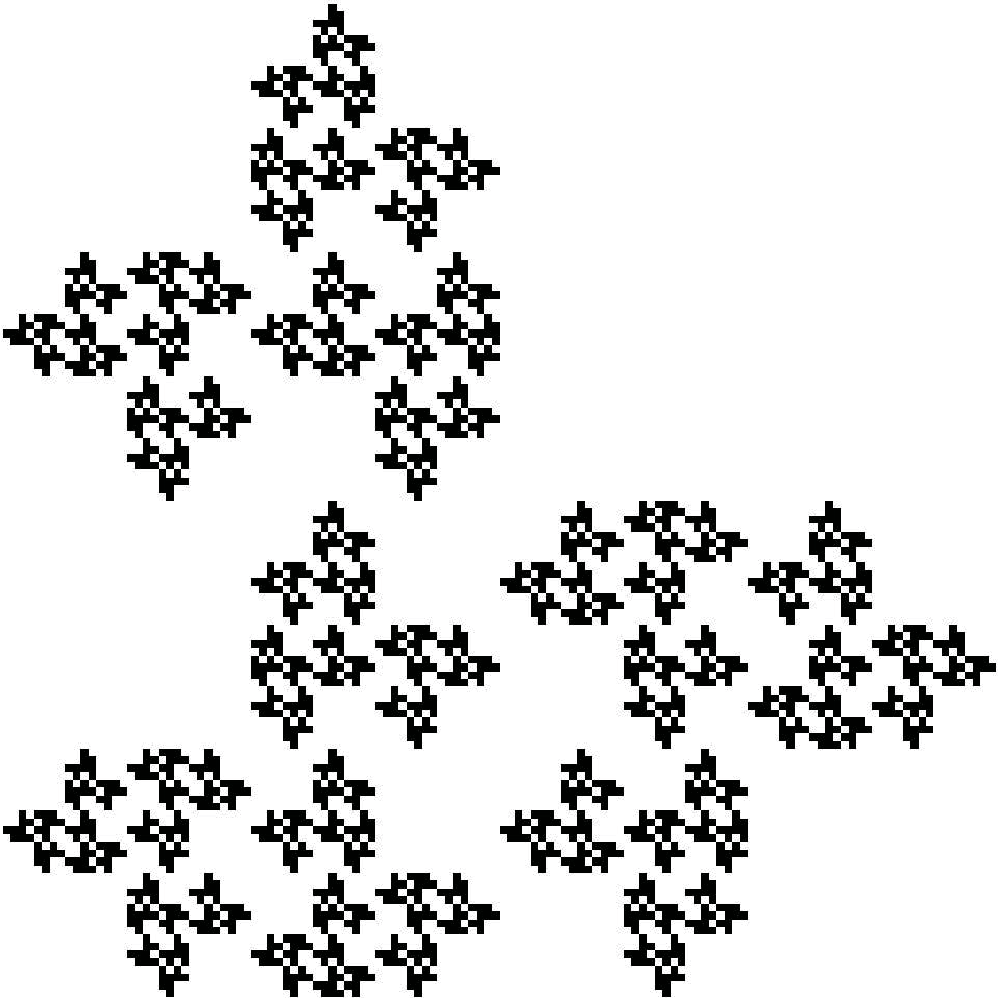}
    \caption{Three sponge-like sets in $\R^2$}
    %The carpet-like fractal corresponding to Example~\ref{exa:carpet1}
    \label{fig:carpet1}
\end{figure}

Over the last decade, sponge-like sets have been studied intensively, especially in cases where all of the contracting maps in the IFS are orientation preserving. Unfortunately, even for this specific class of self-similar sets, many basic topological properties are still far from clear to us. Criteria for generalized Sierpi\'nski carpets that are totally disconnected or have only finitely many connected components are given in~\cite{LLR13} and~\cite{Xiao21}, respectively. Partial results on the quasisymmetric and Lipschitz classification of Sierpi\'nski sponges can be found in~\cite{BM13,BM20,RW17,Xi20} and references therein. In a recent paper~\cite{DLRWX21} joint with Dai, Luo and Wang, the authors provided characterizations of connected Sierpi\'nski carpets with local cut points.

For general cases allowing rotations and reflections, the topology becomes involved and there are few existing studies. For example, it is simple to observe that there are $|\mathcal{O}_d|^{|\I|}$ different IFSs provided $\I$ is fixed, but the enumeration of distinct attractors is difficult since many of them coincide. This problem was solved in~\cite{FO07} by Falconer and O'Conner using group theory. In the case when $d=2$, Fraser~\cite{Fra12} investigated the self-affine generalization of our sponge-like sets and calculated their packing and box dimensions under some open set type conditions.

In this paper, we are able to determine the connectedness of any given sponge-like set $F$. By Hata's criterion, it suffices to check whether $\varphi_i(F)\cap\varphi_j(F)=\varnothing$ for all $i,j\in\mathcal{I}$. Note that $\varphi_i(F)$ (resp. $\varphi_j(F)$) is just a scaled copy of $O_i(F)$ (resp. $O_j(F)$). Our key idea is to regard $(O(F))_{O\in\mathcal{O}(d)}$ as the attractor of some special graph-directed system (see Lemma~\ref{lem:cantortypeatt}) and to determine whether its coordinates intersect with each other. Let us pause here to give a brief introduction to graph-directed sets and related work.

Graph-directed sets can be regarded as a generalization of self-similar sets, and the standard process of generating them is as follows. Let $G=G(\V,\E)$ be a directed graph where $\V=\{1,2,\ldots,n\}$ is the vertex set and $\E$ is the edge set. Assume further that for each vertex in $\V$, there is at least one edge starting from it. A graph-directed IFS (iterated function system) in $\R^d$ is a finite family $\Phi=\{\vp_e: e\in\E\}$ consisting of contracting similarities from $\R^d$ to $\R^d$ indexed by edges in $\E$. It is well known (see~\cite{Fal97} or~\cite{MW88}) that there is a unique $n$-tuple of non-empty compact sets $(K_1,\ldots,K_n)$ such that
\[
    K_i = \bigcup_{j=1}^n\bigcup_{e\in\E_{i,j}} \vp_e(K_j),
\]
where $\E_{i,j}$ denotes the collection of edges in $\E$ from $i$ to $j$. Such a tuple is usually called the \emph{graph-directed attractor} associated with the graph-directed system $(G,\Phi)$, and the corresponding \emph{graph-directed set} often refers to the union $\bigcup_{i=1}^n K_i$.

In~\cite{MW88}, Mauldin and Williams calculated the Hausdorff dimension of graph-directed sets under an open set type condition. It is also shown that the corresponding Hausdorff measure of the set is positive and $\sigma$-finite. Later, Das and Ngai~\cite{DN04} developed algorithms to compute the dimension under some weak separation condition (called the finite type condition). Xiong and Xi~\cite{XX09} studied the Lipschitz classification of dust-like graph-directed sets. For further work, see~\cite{EG99,EM92,Far19,Ols94,Ols11,Wang97,WX10} and references therein.

We may always assume safely that for each $\E_{i,j}$, $\vp_e\neq\vp_{e'}$ whenever $e,e'\in\E_{i,j}$ are distinct. %Otherwise, the graph-directed system $(G',\Phi')$ generates the same attractor as $(G,\Phi)$ does, where $G'$ is the subgraph of $G$ after deleting $e'$ and $\Phi'=\Phi\setminus\{\vp_{e'}\}$, and in this case it suffices to consider $(G',\Phi')$ instead of $(G,\Phi)$. 
We also remark that for every vertex $v$ of every graph that appears in this paper, one is able to travel from $v$ directly to itself if and only if there is a loop at $v$.

Let $(K_1,\ldots,K_n)$ be a graph-directed attractor. Our task is to determine whether $K_i$ intersects $K_j$ for all $1\leq i,j\leq n$. In this paper, we solve this intersection problem for a special class of graph-directed sets, which is enough to settle the connectedness of the aforementioned sponge-like sets. Although our approach remains valid in a more general setting (see Section 4 for details), the treatment of the following class, namely Cantor-type graph-directed attractors, would be adequate to demonstrate the idea.

\begin{definition}\label{de:gds}
    Let $(K_1,\ldots,K_n)$ be a graph-directed attractor in $\R^d$ associated with $G=G(\V,\E)$ and $\Phi=\{\vp_e: e\in\E\}$. We say that the attractor $(K_1,\ldots,K_n)$ or the IFS $\Phi$ is of \emph{Cantor-type} if for each $e\in\E$, $\vp_e$ is a contracting similarity on $\R^d$ of the form $\vp_e(x)=N^{-1}(x+t_e)$, where $N\geq 2$ is an integer and $t_e\in\{0,1,\ldots,N-1\}^d$.
\end{definition}

Let $(K_1,\ldots,K_n)$ be a Cantor-type graph-directed attractor in $\R^d$. The following geometric iteration process to obtain the attractor is standard, and we present here a sketch of proof just for completeness.

\begin{lemma}\label{lem:approximation}
    Let $Q_{i,0}:=[0,1]^d$ for $1\leq i\leq n$ and recursively define
    \[
        Q_{i,k} := \bigcup_{j=1}^n\bigcup_{e\in\E_{i,j}} \vp_e(Q_{j,k-1}), \quad k\geq 1, 1\leq i\leq n.
    \]
    Then for each $i$, $\{Q_{i,k}\}_{k=1}^\infty$ forms a decreasing sequence of compact sets and $K_i=\bigcap_{k=1}^\infty Q_{i,k}$.
\end{lemma}
We name $Q_{i,k}$ the \emph{level-$k$ approximation} of $K_i$.
\begin{proof}
    It is easy to see that
    \[
        \vp_e(Q_{i,0}) = \vp_e([0,1]^d) \subset [0,1]^d = Q_{i,0}, \quad \forall e\in\E, 1\leq i\leq n,
    \]
    from which one can show by induction that $\{Q_{i,k}\}_{k=1}^\infty$ is decreasing. Note that $Q_{i,k}$ is a finite union of closed cubes and hence is compact. Finally, it follows from
    \[
        \bigcap_{k=1}^\infty Q_{i,k} = \bigcap_{k=1}^\infty \Big( \bigcup_{j=1}^n\bigcup_{e\in\E_{i,j}} \vp_e(Q_{j,k-1}) \Big) = \bigcup_{j=1}^n\bigcup_{e\in\E_{i,j}}  \vp_e\Big( \bigcap_{k=1}^\infty Q_{j,k-1} \Big)
    \]
    and the uniqueness of the graph-directed attractor that $K_i=\bigcap_{k=1}^\infty Q_{i,k}$. Note that the last equality above follows from the monotonicity of $\{Q_{i,k}\}_k$.
\end{proof}  

For convenience and clarity, we will write $\V_n:=\{1,\ldots,n\}$ throughout this paper. The desired solution to the aforementioned problem is achieved utilizing a readily created graph called the \emph{intersection graph}. This graph has a vertex set consisting of pairs $(i,j)$ which are joined (by edges or dashed edges) roughly according to whether the images of the unit cube have common faces, see Definition~\ref{de:intersectgraph}.

Our criterion is:

\begin{theorem}\label{thm:main1}
    Let $(K_1,\ldots,K_n)$ be a Cantor-type graph-directed attractor in $\R^d$. For every pair of distinct $i,j\in\V_n$, $K_i\cap K_j\neq\varnothing$ if and only if there exists, in the associated intersection graph, either a terminated finite walk or an infinite solid walk that starts from $(i,j)$.
\end{theorem}

We remark that throughout this paper, a walk always allows repeated vertices and edges. It is also worthwhile considering the intersection problem from a different perspective: \emph{Is there a constant $C$ such that for every pair of $i\neq j$, $K_i\cap K_j\neq\varnothing$ whenever $Q_{i,C}\cap Q_{j,C}\neq\varnothing$?} The answer is actually affirmative. Write

\begin{equation}\label{eq:constant}
    c(n,d):=(d-1)n^2+\frac{n^2+n}{2}+d-1, \quad d\geq 1.
\end{equation}
By analyzing walks in the associated intersection graph, we have the following result.

\begin{theorem}\label{thm:main2}
   Let $(K_1,\ldots,K_n)$ be a Cantor-type graph-directed attractor in $\R^d$. Assume that $i,j\in\V_n$ are distinct. Then $K_i\cap K_j\neq\varnothing$ if and only if $Q_{i,c(n,d)} \cap Q_{j,c(n,d)}\neq\varnothing$.
\end{theorem}

With the aid of Theorem~\ref{thm:main1}, we can draw the Hata graph associated with any given sponge-like set and thus determine whether it is connected; see Section 3 for details. As a corollary of Theorem~\ref{thm:main2}, we also show that it suffices to check only a finite number (independent of $N$) of its geometric approximations.

\begin{theorem}\label{thm:main3}
    Let $F=F(d,N,\I)$ be a sponge-like set in $\R^d$. Then $F$ is connected if and only if $F_{C(d)}$ is connected, where $C(d):=c(d!2^{d}+2,d)$ and $c(\cdot,\cdot)$ is as in~\eqref{eq:constant}.
\end{theorem}

%$F_{C(d)+1}$ is connected, where $C(d):=c(d!2^{d}+2,d)-1$ %and $c(\cdot,\cdot)$ is as in~\eqref{eq:constant}.

The paper is organized as follows. In Section 2, we construct the intersection graph and prove Theorem~\ref{thm:main1}. In Section 3, we prove Theorem~\ref{thm:main2} and discuss the sharpness of the constant $c(n,d)$. In Section 4, we present the method of plotting Hata graphs associated with sponge-like sets and prove Theorem~\ref{thm:main3}. Section 3.2 is devoted to discussions on possible improvements of the constant $C(d)$ in Theorem~\ref{thm:main3} under suitable conditions. Finally, we discuss general settings under which our approach remains applicable in the last section.

%-------------------------------------------
\section{The intersection problem I: Creating auxiliary graphs}

Throughout this section, $(K_1,\ldots,K_n)$ is presumed to be a fixed graph-directed attractor in $\R^d$ which is of Cantor-type, and the notations $G=G(\V_n,\E)$, $\E_{i,j}$, $\vp_e$, $t_e$, $Q_{i,k}$, etc. are as in Section 1. Here is a concrete example.

\begin{example}\label{exa:1}
    Let $G$ be the directed graph as in Figure~\ref{fig:digraphofexa}. Set $d=1$, $N=4$ and
    \[
        \vp_{e_1}(x)=\vp_{e_3}(x) = \frac{x}{4},\, \vp_{e_2}(x)=\frac{x}{4}+\frac{1}{4},\, \vp_{e_4}(x) = \frac{x}{4}+\frac{3}{4}.
    \]
    So $t_{e_1}=t_{e_3}=0$, $t_{e_2}=1$ and $t_{e_4}=3$. Clearly, $\{\vp_{e_i}: 1\leq i\leq 4\}$ is of Cantor-type.
\end{example}

\begin{figure}[htbp]
    \centering
        \begin{tikzpicture}[>=stealth]
            \coordinate[label = $1$] (1) at (0, 0);
            \node at (1)[circle,fill,inner sep=2pt]{};
            \coordinate[label = $2$] (2) at (3, 0);
            \node at (2)[circle,fill,inner sep=2pt]{};
            \draw[-{>[flex=1]},thick] (-0.1,0.12) arc (20:340:.3cm);
            \draw[-{>[flex=1]},thick] (3.1,0.12) arc (160:-160:.3cm);
            \draw[-{>[flex=1]},thick] (0.15,0.12) arc (120:60:2.7cm);
            \draw[-{>[flex=1]},thick] (2.85,-0.12) arc (-60:-120:2.7cm);
            %\draw[->,thick] (0.2,0) to (2.8,0);
            \node at (-0.9,0)[circle,inner sep=2pt]{$e_2$};
            \node at (3.9,0)[circle,inner sep=2pt]{$e_3$};
            \node at (1.5,0.7)[circle,inner sep=2pt]{$e_1$};
            \node at (1.5,-0.7)[circle,inner sep=2pt]{$e_4$};
        \end{tikzpicture}
        \caption{Directed graph in Example~\ref{exa:1}}
        \label{fig:digraphofexa}
\end{figure}

For $k\geq 1$ and $i\in\V_n$, denote $\E_i^k$ to be the collection of walks in the directed graph $G$ which starts from $i$ and has length $k$. For $\mathbf{e}=(e_1,\ldots,e_k)\in\E_i^k$, denote by $\ter(\mathbf{e})$ the terminal vertex of $\mathbf{e}$ and write $\vp_{\mathbf{e}}:=\vp_{e_1}\circ\cdots\circ\vp_{e_k}$.

\begin{lemma}\label{lem:lengthkwalk}
    Let $k\geq 1$ and let $i\in\V_n$. Then for each $t\geq 0$,
    \[
        Q_{i,k+t} = \bigcup_{\mathbf{e}\in\E_i^k} \vp_{\mathbf{e}}(Q_{\ter(\mathbf{e}),t}).
    \]
    As a result, $K_i = \bigcup_{\mathbf{e}\in\E_i^k} \vp_{\mathbf{e}}(K_{\ter(\mathbf{e})})$.
\end{lemma}
\begin{proof}
    The proof is straightforward: Just note that
    \begin{align*}
        Q_{i,k+t} &= \bigcup_{j_1=1}^n\bigcup_{e_1\in\E_{i,j_1}}\vp_{e_1}(Q_{j_1,k+t-1}) \\
        &= \bigcup_{j_1=1}^n\bigcup_{e_1\in\E_{i,j_1}}\vp_{e_1}\Big( \bigcup_{j_2=1}^n\bigcup_{e_2\in\E_{j_1,j_2}}\vp_{e_2}(Q_{j_2,k+t-2}) \Big) \\
        &= \cdots = \bigcup_{\mathbf{e}\in\E_i^k} \vp_{\mathbf{e}}(Q_{\ter(\mathbf{e}),t}).
    \end{align*}
    Since $\{Q_{j,k}\}_k$ is decreasing for all $j$,
    \begin{align*}
        K_i = \bigcap_{t=0}^\infty Q_{i,k+t} = \bigcap_{t=0}^\infty\bigcup_{\mathbf{e}\in\E_i^k} \vp_{\mathbf{e}}(Q_{\ter(\mathbf{e}),t}) = \bigcup_{\mathbf{e}\in\E_i^k} \vp_{\mathbf{e}}\Big( \bigcap_{t=0}^\infty Q_{\ter(\mathbf{e}),t} \Big) = \bigcup_{\mathbf{e}\in\E_i^k} \vp_{\mathbf{e}}(K_{\ter(\mathbf{e})}).
    \end{align*}
\end{proof}

In particular, $Q_{i,k}=\bigcup_{\mathbf{e}\in\E_i^k}\vp_{\mathbf{e}}([0,1]^d)$. For narrative convenience, we shall call any cube in $\bigcup_{i=1}^n \{\vp_{\mathbf{e}}([0,1]^d):\mathbf{e}\in\E_i^k\}$ a \emph{level-$k$} cube. Note that every $Q_{i,k}$ is the union of a finite number of level-$k$ cubes. 

Recall that for any given polytope $P\subset\R^d$, a \emph{face} of $P$ is any non-empty set of the form $P\cap \{x\in\R^d: a\cdot x=b\}$, where $a\in\R^d$, $b\in \R$ and the inequality $a\cdot x\leq b$ is valid for all $x\in P$ (see~\cite{Zie95}). Moreover, the \emph{dimension} of a face is just the dimension of its affine hull. For example, faces of the unit cube $[0,1]^d$ of dimensions $0$ and $1$ are vertices and edges of $[0,1]^d$, respectively. Since the graph-directed attractor is of Cantor-type, it is not hard to see that every level-$k$ cube can be written as
\[
    \Big[ \frac{p_1}{N^k},\frac{p_1+1}{N^k} \Big] \times \Big[ \frac{p_2}{N^k},\frac{p_2+1}{N^k} \Big] \times\cdots\times \Big[ \frac{p_d}{N^k},\frac{p_d+1}{N^k} \Big]
\] for some integers $0\leq p_1,\ldots,p_d\leq N^k-1$. Consequently, the intersection of any pair of level-$k$ cubes is just the largest common face they share.

%----------------------------------------------------
\subsection{Examination of faces of the unit cube}
Since all of the translation vectors belong to $\{0,1,\ldots,N-1\}^d$, if two level-$k$ cubes intersect then they have a common face. So a preliminary question to the intersection problem is: \emph{Given any face $P$ of $[0,1]^d$ with $0\leq \dim P\leq d-1$, which of $K_1,\ldots,K_n$ intersects it?} The following lemma and its corollaries help to determine which parts of $K_i$ touch $P$.

\begin{lemma}\label{lem:face}
    Let $P$ be a face of $[0,1]^d$ with $0\leq \dim P\leq d$ and let $\vp(x)=N^{-1}(x+t)$ where $t=(t_1,\ldots,t_d)\in\{0,1,\ldots,N-1\}^d$. We have for all $E\subset[0,1]^d$ that $P\cap\vp(E)\subset\vp(P\cap E)$. Furthermore, if $P\cap\vp(E)$ is non-empty then it equals $\vp(P\cap E)$. 
\end{lemma}
\begin{proof}
    If $\dim P=d$ then $P$ is the unit cube itself and there is nothing to prove. So it suffices to consider cases when $\dim P\leq d-1$. Letting $s:=d-\dim P$, we can find a $0$-$1$ vector $\alpha=(a_1,\ldots,a_s)$ and a sequence $1\leq m_1<m_2<\cdots<m_s\leq d$ such that
    \begin{equation}\label{eq:P-def}
        P = \{(x_1,\ldots,x_d)\in[0,1]^d: x_{m_k}=a_k, 1\leq k\leq s\}.
    \end{equation}

    If $P\cap\vp(E)=\varnothing$ then the inclusion clearly holds, so we may assume that $P\cap\vp(E)\neq\varnothing$. Choosing any $y\in P\cap\vp(E)$, there is some $(x_1,\ldots,x_d)\in E$ such that $y=\vp(x_1,\ldots,x_d)=N^{-1}(x_1+t_{1},\ldots,x_d+t_{d})$. Since $y\in P$, $N^{-1}(x_{m_k}+t_{m_k})=a_k$ for $1\leq k\leq s$. If $a_k=0$ then $x_{m_k}+t_{m_k}=0$. So $x_{m_k}=t_{m_k}=0=a_k$. If $a_k=1$ then $x_{m_k}+t_{m_k}=N$. Since $t_{m_k}\leq N-1$, we have $x_{m_k}=1=a_k$ and $t_{m_k}=N-1$. We conclude that $x_{m_k}=a_k$ for $1\leq k\leq s$ so $(x_1,\ldots,x_d)\in P$. Thus $y=\vp(x_1,\ldots,x_d) \in \vp(P\cap E)$. Since $y$ is arbitrarily chosen, $P\cap\vp(E)\subset \vp(P\cap E)$.

    On the other hand, let $z=(z_1,\ldots,z_d)\in P\cap E$. It follows from the definition of $P$ that $z_{m_k}=a_k$ for $1\leq k\leq s$. Since $P\cap \vp(E)\not=\varnothing$, from the above argument, $t_{m_k}=a_k$ if $a_k=0$ and $t_{m_k}=N-1$ if $a_k=1$. Thus $z_{m_k}+t_{m_k}=Na_k$ for $1\leq k\leq s$ so that
    \[
        N^{-1}(z_{m_k}+t_{m_k}) = a_k.
    \] 
    This implies that $\vp(z)\in P$ and hence $\vp(P\cap E)\subset P\cap\vp(E)$. So $\vp(P\cap E)=P\cap\vp(E)$.
\end{proof}

\begin{corollary}\label{lem:fixpt}
    Let $\alpha$ be a vertex of $[0,1]^d$ and let $e\in\E$. If $\alpha\in\vp_e([0,1]^d)$ then $\alpha=\vp_e(\alpha)$. Moreover, $\vp_e(x)\neq\alpha$ whenever $x\neq\alpha$.
\end{corollary}
\begin{proof}
    Since $\alpha$ is a $0$-dimensional face of $[0,1]^d$, Lemma~\ref{lem:face} implies that 
    \[
        \{\alpha\}=\{\alpha\}\cap\vp_e([0,1]^d)=\{\vp_{e}(\alpha)\}
    \]
    and hence $\alpha=\vp_e(\alpha)$. The second statement clearly holds since $\vp_e$ is injective.
\end{proof}

The following corollary will be used later.

\begin{corollary}\label{cor:face}
    Let $\vp_1(x)=N^{-1}(x+t_1),\ldots,\vp_m(x)=N^{-1}(x+t_m)$ be contracting similarities with $t_1,\ldots,t_m\in\{0,1,\ldots,N-1\}^d$. Let $P$ be a face of $[0,1]^d$ with $1\leq \dim P\leq d$. If $\vp_i([0,1]^d)\cap P\neq\varnothing$ for all $1\leq i\leq m$ then
    \[
        \big( \vp_1\circ\cdots\circ\vp_m([0,1]^d) \big) \cap P = \big( \vp_1\circ\cdots\circ\vp_m(P) \big) \cap P.
    \]
\end{corollary}
\begin{proof}
    We will prove this by induction on $m$. When $m=1$, this follows directly from Lemma~\ref{lem:face}. Suppose the result holds for $1\leq m\leq k$. Then
    \begin{align*}
        \big( \vp_1\circ\cdots\circ\vp_{k+1}([0,1]^d) \big) \cap P &= \big( \vp_1\circ\cdots\circ\vp_{k+1}([0,1]^d) \big) \cap \big( \vp_1\circ\cdots\circ\vp_{k}([0,1]^d) \big) \cap P \\
        &= \big( \vp_1\circ\cdots\circ\vp_{k+1}([0,1]^d) \big) \cap \big( \vp_1\circ\cdots\circ\vp_{k}(P) \big) \cap P \\
        &= \big( \vp_1\circ\cdots\circ\vp_{k}(\vp_{k+1}([0,1]^d)\cap P) \big) \cap P \\
        &= \big( \vp_1\circ\cdots\circ\vp_k\circ\vp_{k+1}(P) \big) \cap P,
    \end{align*}
    as desired.
\end{proof}

Let $P$ be a face of $[0,1]^d$ with $0\leq\dim P\leq d-1$. To determine which of $K_1, \ldots, K_n$ intersects $P$, we will draw an auxiliary directed graph $G_P$ as follows. 

\begin{definition}[Auxiliary graph]
    The vertex set of $G_P$ is the index set $\V_n=\{1,2,\ldots,n\}$. For $i,j\in\V_n$, we add an edge from $i$ to $j$ whenever $P\cap\bigcup_{e\in\E_{i,j}}\vp_e([0,1]^d)\neq\varnothing$. 
\end{definition}

Figure~\ref{fig:vertexgraofexa} depicts the graphs $G_{\{0\}}$ and $G_{\{1\}}$ associated with the attractor in Example~\ref{exa:1} (where the subscripts $0,1$ stand for the endpoints of $[0,1]$).

\begin{figure}
    \centering
    \subfloat[The graph $G_{\{0\}}$]
    {
        \label{fig: F1}
        \begin{minipage}[b]{145pt}
        \centering
        \begin{tikzpicture}[>=stealth]
            \coordinate[label = $1$] (1) at (0, 0);
            \node at (1)[circle,fill,inner sep=2pt]{};
            \coordinate[label = $2$] (2) at (3, 0);
            \node at (2)[circle,fill,inner sep=2pt]{};
            \draw[-{>[flex=1]},thick] (3.1,0.12) arc (160:-160:.3cm);
            \draw[-{>[flex=1]},thick] (0.15,0.12) arc (120:60:2.7cm);
        \end{tikzpicture}
      \end{minipage}
    }
    \subfloat[The graph $G_{\{1\}}$]
    {
        \begin{minipage}[b]{145pt}
        \centering
        \begin{tikzpicture}[>=stealth]
            \coordinate[label = $1$] (1) at (0, 0);
            \node at (1)[circle,fill,inner sep=2pt]{};
            \coordinate[label = $2$] (2) at (3, 0);
            \node at (2)[circle,fill,inner sep=2pt]{};
            \draw[-{>[flex=1]},white] (-0.1,0.12) arc (20:340:.3cm);
            \draw[-{>[flex=1]},white] (3.1,0.12) arc (160:-160:.3cm);
            \draw[-{>[flex=1]},thick] (2.85,0.12) arc (60:120:2.7cm);
        \end{tikzpicture}
        \end{minipage}
    }
    \caption{Auxiliary graphs in Example~\ref{exa:1}}
    \label{fig:vertexgraofexa}
\end{figure}

The information provided by the graph $G_P$ is revealed in the next two lemmas.

\begin{lemma}\label{lem:0ink1}
    Let $P$ be a face of $[0,1]^d$ with $0\leq\dim P\leq d-1$ and let $i\in\V_n$. For $k\geq 1$, if there is a walk of length $k$ in the graph $G_P$ which starts from $i$, then $P\cap Q_{i,k}\neq\varnothing$.
\end{lemma}
\begin{proof}
    This can be proved by induction on $k$. Suppose $k=1$ and there is an edge from $i$ to some $j$ in $G_P$. Then there is some $e\in\E_{i,j}$ such that $P\cap\vp_e([0,1]^d)\neq\varnothing$. By Lemma~\ref{lem:face},
    \[
        \vp_e(P) = P\cap\vp_{e}([0,1]^d)= P\cap\vp_{e}(Q_{j,0}) \subset P\cap\bigcup_{t=1}^n\bigcup_{e\in\E_{i,t}} \vp_e(Q_{t,0}) = P\cap Q_{i,1}.
    \]

    Suppose the lemma holds for $1\leq k\leq m$ and let $i\to i_1\to\cdots\to i_m\to i_{m+1}$ be a walk in $G_P$ of length $m+1$. Then $i_1\to i_2\to\cdots\to i_{m+1}$ is a walk of length $m$, so we have by the induction hypothesis that $P\cap Q_{i_1,m}\neq\varnothing$. Since there is an edge from $i$ to $i_1$, we can find some $e_1\in\E_{i,i_1}$ such that $P\cap\vp_{e_1}([0,1]^d)\neq\varnothing$. It then follows from Lemma~\ref{lem:face} that
    \begin{align*}
        P\cap Q_{i,m+1} &= P\cap\bigcup_{t=1}^n\bigcup_{e\in\E_{i,t}}(Q_{t,m}) \\
        &\supset P\cap\vp_{e_1}(Q_{i_1,m}) \\
        &= (P\cap\vp_{e_1}([0,1]^d))\cap\vp_{e_1}(Q_{i_1,m}) \\
        &= \vp_{e_1}(P)\cap\vp_{e_1}(Q_{i_1,m}) = \vp_{e_1}(P\cap Q_{i_1,m}) \neq\varnothing.
    \end{align*}
    This completes the induction.
\end{proof}

\begin{lemma}\label{lem:0ink2}
    Let $P$ be a face of $[0,1]^d$ with $0\leq\dim P\leq d-1$ and let $i\in\V_n$. For $k\geq 1$, if $P\cap Q_{i,k}\neq\varnothing$ then we can find a walk of length $k$ in $G_P$ which starts from $i$.
\end{lemma}
\begin{proof}
    Since
    \[
        \varnothing\neq P\cap Q_{i,k}=P\cap\bigcup_{j=1}^n\bigcup_{e\in\E_{i,j}} \vp_e(Q_{j,k-1}),
    \]
    there is some $j_1$ and $e_1\in\E_{i,j_1}$ such that $P\cap \vp_{e_1}(Q_{j_1,k-1})\neq\varnothing$. Since $Q_{j_1,k-1}\subset[0,1]^d$, $P\cap \vp_{e_1}([0,1]^d)\neq\varnothing$. By definition, there is an edge from $i$ to $j_1$ in the graph $G_P$. Moreover, we see by Lemma~\ref{lem:face} that $P\cap\vp_{e_1}(Q_{j_1,k-1}) =\vp_{e_1}(P\cap Q_{j_1,k-1})$, so $P\cap Q_{j_1,k-1}\neq\varnothing$. Analogously, there is an edge in $G_P$ from $j_1$ to some $j_2$ with $P\cap Q_{j_2,k-2}\neq\varnothing$. Continuing in this manner, we obtain a walk $i\to j_1\to j_2\to\cdots\to j_k$.
\end{proof}

From these facts, it is easy to check from the graph $G_P$ whether $K_i\cap P$ is empty for all $i$.

\begin{corollary}\label{cor:01}
    Let $P$ be a face of $[0,1]^d$ with $0\leq\dim P\leq d-1$. For $i\in\V_n$, $P\cap K_i\neq\varnothing$ if and only if there are arbitrarily long walks in the graph $G_P$ which start from $i$.
\end{corollary}
\begin{proof}
    By Lemma~\ref{lem:approximation}, $P\cap K_i\neq\varnothing$ if and only if $P\cap Q_{i,k}\neq\varnothing$ for all $k\geq 1$. Then the result is an immediate consequence of Lemmas~\ref{lem:0ink1} and~\ref{lem:0ink2}.
\end{proof}

In particular, for the graph-directed attractor $(K_1,K_2)$ in Example~\ref{exa:1}, we have $0\in K_1$, $0\in K_2$ while $1\notin K_1$ and $1\notin K_2$.

%------------------------------------
\subsection{The intersection problem}

Now we turn to the intersection problem of $(K_1,\ldots,K_n)$. It is proved that to check whether $K_i\cap K_j$ is empty, it suffices to examine two types of walks in the directed graph defined as follows.

\begin{definition}[Intersection graph]\label{de:intersectgraph}
    The vertex set is $\{(i,j):i,j\in\V_n\}$. The edge set is defined as follows. For any vertex $(i,j)$ with $i\neq j$:
    \begin{enumerate}
        \item There is a solid edge from $(i,j)$ to some $(i',j')$ if we can find $e\in\E_{i,i'}$ and $e'\in\E_{j,j'}$ such that $\vp_e=\vp_{e'}$;
        \item There is a dashed edge from $(i,j)$ to some $(i',j')$ if we can find $e\in\E_{i,i'}$ and $e'\in\E_{j,j'}$ such that $\vp_e([0,1]^d)\cap\vp_{e'}([0,1]^d)$ is a common $s$-dimensional face of these two cubes with $0\leq s\leq d-1$.
    \end{enumerate}
\end{definition}

It is also easy to observe a symmetric property of the intersection graph: If there is an edge from $(i,j)$ to $(i',j')$ then there is an edge from $(j,i)$ to $(j',i')$ of the same type (i.e., solid or dashed). We remark that there can be multiple edges from one vertex to another, but at most one edge of each type. See Figure~\ref{fig:intergraph} for the intersection graph associated with the graph-directed system in Example~\ref{exa:1}.

\begin{figure}[htbp]
    \centering
    \begin{tikzpicture}[>=stealth]
        %\coordinate (1) at (0, 0);
        %\fill (0,0) circle (0.1) node[above]{$(1,1)$};
        \fill (0,0) circle (0.1) node[below=1pt]{$(2,1)$};
        \fill (2,0) circle (0.1) node[below=1pt]{$(2,2)$};
        \fill (0,2) circle (0.1) node[above=1pt]{$(1,1)$};
        \fill (2,2) circle (0.1) node[above=1pt]{$(1,2)$};
        \draw[->,thick] (0.15,0) to (1.85,0);
        %\draw[->,thick] (0.15,1.85) to (1.85,0.15);
        \draw[->,thick] (2,1.85) to (2,0.15);
        \draw[->,thick,dashed] (1.84,2) arc (255:-75:0.6cm and 0.4cm);
        \draw[->,thick,dashed] (0.16,0) arc (-75:255:0.6cm and 0.4cm);
        %\node[below,circle,fill] {$f(x) =x$};
        %\node at (1)[circle,fill,inner sep=2pt]{};
        %\node[above] {$(1,1)$};
    \end{tikzpicture}
    \caption{Intersection graph in Example~\ref{exa:1}}
    \label{fig:intergraph}
\end{figure}

\begin{lemma}\label{lem:blackedge}
    If there is a solid edge from $(i,j)$ to $(i',j')$, then $K_i\cap K_j$ contains a scaled copy of $K_{i'}\cap K_{j'}$.
\end{lemma}
\begin{proof}
    The assumption gives us $e\in\E_{i,i'}$ and $e'\in\E_{j,j'}$ such that $\vp_e=\vp_{e'}$. Thus
    \begin{align*}
        K_i \cap K_j &= \Big( \bigcup_{t=1}^n\bigcup_{w\in\E_{i,t}} \vp_w(K_t) \Big) \cap \Big( \bigcup_{t=1}^n\bigcup_{w\in\E_{j,t}} \vp_w(K_t) \Big) \\
        &\supset \vp_e(K_{i'}) \cap \vp_{e'}(K_{j'}) = \vp_e(K_{i'}\cap K_{j'}).
    \end{align*}
    Since $\vp_e$ is a contracting similarity, this completes the proof.
\end{proof}

\begin{proposition}\label{prop:gds2}
    Let $i,j\in\V_n$ be distinct. If there exists an infinite solid walk in the intersection graph that starts from $(i,j)$, then $K_i\cap K_j\neq\varnothing$.
\end{proposition}
By a \emph{solid walk} we mean a walk in which all of the edges are solid.
\begin{proof}
    Let $(i,j)\to(i_1,j_1)\to(i_2,j_2)\to\cdots$ be such an infinite walk. Since all edges are solid, we can find for all $k\geq 1$ that $e_k\in\E_{i_k,i_{k+1}}$, $e'_k\in\E_{j_k,j_{k+1}}$ with $\vp_{e_k}=\vp_{e'_k}$. 
    Note that for all $m\geq 1$, $i_1\xrightarrow{e_1}i_2\xrightarrow{e_2}\cdots\xrightarrow{e_m}i_{m+1}$ is a walk in $G$ of length $m$. By Lemma~\ref{lem:lengthkwalk}, 
    \[
        \vp_{e_1}\circ\cdots\circ\vp_{e_m}([0,1]^d) = \vp_{e_1}\circ\cdots\circ\vp_{e_m}(Q_{m+1,0}) \subset Q_{i_1,m}.
    \]
    Similarly, $\vp_{e'_1}\circ\cdots\circ\vp_{e'_m}([0,1]^d)\subset Q_{j_1,m}$. Since $\vp_{e_k}=\vp_{e'_k}$ for all $k$, we see that 
    \[
        \vp_{e_1}\circ\cdots\circ\vp_{e_m}([0,1]^d) \subset Q_{i_1,m}\cap Q_{j_1,m}, \quad m\geq 1.
    \]
    As a result,
    \begin{align*}
        K_{i_1}\cap K_{j_1} &= \Big( \bigcap_{m=1}^\infty Q_{i_1,m} \Big) \cap \Big( \bigcap_{m=1}^\infty Q_{j_1,m} \Big) \\
        &= \bigcap_{m=1}^\infty Q_{i_1,m}\cap Q_{j_1,m} \supset \bigcap_{m=1}^\infty \vp_{e_1}\circ\cdots\circ\vp_{e_m}([0,1]^d),
    \end{align*}
    which is a singleton since $\{\vp_{e_1}\circ\cdots\circ\vp_{e_m}([0,1]^d)\}_{m=1}^\infty$ is a nested sequence. In particular, $K_{i_1}\cap K_{j_1}\neq\varnothing$ and we immediately have by Lemma~\ref{lem:blackedge} that $K_i\cap K_j\neq\varnothing$.
\end{proof}

Our solution to the intersection problem is based on the detection of infinite solid walks and terminated finite walks in the intersection graph. The latter are defined as follows.

\begin{definition}\label{de:terminated}
    Let $(i,j),(i',j')$ be two vertices in the intersection graph. For an edge from $(i,j)$ to $(i',j')$, we call it \emph{terminated} if one of the following happens:
    \begin{enumerate}
        \item The edge is solid and $i'=j'$;
        \item The edge is dashed and there are $e\in\E_{i,i'}, e'\in\E_{j,j'}$ such that $\vp_e([0,1]^d)\cap\vp_{e'}([0,1]^d)$ is a common lower dimensional face of these two cubes and $\vp_e(K_{i'})\cap\vp_{e'}(K_{j'})\neq\varnothing$.
    \end{enumerate}
    Generally, a finite walk of length $m\geq 2$ is called \emph{terminated} if the first $(m-1)$ edges are all solid while the last one is terminated.
\end{definition}

Readers might have noticed that, unlike drawing the intersection graph, determining terminated edges is usually a non-trivial task. While terminated solid edges are easy to check, determining whether a given dashed edge is terminated requires much more information (especially when the dimension $d$ is large). In fact, the latter relies on our solution of the intersection problem of attractors in the lower dimensional spaces $\R,\R^2,\ldots,\R^{d-1}$. See Section 2.3 for a detailed explanation of this induction method.

\begin{proposition}\label{prop:gds1}
    Let $i,j\in\V_n$ be distinct. If there exists a terminated finite walk in the associated intersection graph that starts from $(i,j)$, then $K_i\cap K_j\neq\varnothing$.
\end{proposition}
\begin{proof}
    Let $(i,j)\to(i_1,j_1)\to\cdots\to(i_m,j_m)$ be a terminated finite walk. When $m=1$, the walk is just a terminated edge. If the edge is solid then $i_1=j_1$. By Lemma~\ref{lem:blackedge}, $K_i\cap K_j$ contains a scaled copy of $K_{i_1}\cap K_{j_1}=K_{i_1}$ and hence is non-empty. If the edge is dashed then there are $e\in\E_{i,i_1}, e'\in\E_{j,j_1}$ such that $\vp_e(K_{i_1})\cap\vp_{e'}(K_{j_1})\neq\varnothing$. Then
    \begin{align*}
        K_i \cap K_j = \Big( \bigcup_{t=1}^n\bigcup_{w\in\E_{i,t}} \vp_w(K_t) \Big) \cap \Big( \bigcup_{t=1}^n\bigcup_{w\in\E_{j,t}} \vp_w(K_t) \Big) \supset \vp_e(K_{i_1}) \cap \vp_{e'}(K_{j_1}) \neq\varnothing.
    \end{align*}

    When $m\geq 2$, $(i_{m-1},j_{m-1})\to(i_m,j_m)$ is a terminated edge and we have seen that $K_{i_{m-1}}\cap K_{j_{m-1}}\neq\varnothing$. Applying Lemma~\ref{lem:blackedge} repeatedly, $K_i\cap K_j$ contains a scaled copy of $K_{i_{m-1}}\cap K_{j_{m-1}}$ and hence is also non-empty.
\end{proof}

The proof of Theorem~\ref{thm:main1} requires another two elementary observations.

\begin{lemma}\label{lem:mainthm1_1}
    Let $i,j\in\V_n$ be distinct and such that $K_i\cap K_j\neq\varnothing$ but there are no terminated finite walks in the intersection graph starting from $(i,j)$. Then for all $m\geq 1$, we can find at least one solid walk which starts from $(i,j)$ and has length $m$.
\end{lemma}
\begin{proof}
    We shall proceed by induction. Fix any pair of such $i,j$. For $e,e'\in\E$, we temporarily say that they are \emph{compatible} if the two cubes $\vp_e([0,1]^d)$ and $\vp_{e'}([0,1]^d)$ have a non-empty intersection. Note that
    \begin{align}
        K_i \cap K_j &= \Big( \bigcup_{p=1}^n\bigcup_{e\in\E_{i,p}} \vp_e(K_p) \Big) \cap \Big( \bigcup_{p=1}^n\bigcup_{e\in\E_{j,p}} \vp_e(K_p) \Big) \notag\\
        &= \bigcup\{\vp_e(K_p)\cap \vp_{e'}(K_{p'}): 1\leq p,p'\leq n, e\in\E_{i,p} \text{ and } e'\in\E_{j,p'} \text{ are compatible}\}. \label{eq:finiteblack1}
    \end{align}
    If $e\in\E_{i,p},e'\in\E_{j,p'}$ are such that $\vp_{e}([0,1]^d)$ and $\vp_{e'}([0,1]^d)$ share a common non-empty lower dimensional face then we must have $\vp_e(K_p)\cap \vp_{e'}(K_{p'})=\varnothing$, since otherwise there is a terminated dashed edge from $(i,j)$ to $(p,p')$. As a result, we have by~\eqref{eq:finiteblack1} that
    \begin{equation}\label{eq:finiteblack2}
        K_i \cap K_j = \bigcup\{\vp_e(K_p)\cap \vp_{e'}(K_{p'}): 1\leq p,p'\leq n, e\in\E_{i,p},e'\in\E_{j,p'},\vp_e=\vp_{e'}\}.
    \end{equation}
    In particular, if $K_i\cap K_j\neq\varnothing$ then there are $p,p'\in\{1,2,\ldots,n\},e\in\E_{i,p},e'\in\E_{j,p'}$ such that $\vp_e=\vp_{e'}$. By definition, there is a solid edge from $(i,j)$ to $(p,p')$. So we establish the lemma in the case when $m=1$.

    Suppose the lemma holds for $1\leq m\leq k$. We will abuse notation slightly by fixing any $i\neq j$ again. If $K_i\cap K_j\neq\varnothing$ then by~\eqref{eq:finiteblack2}, there are $p,p'\in\{1,2,\ldots,n\}, e\in\E_{i,p},e'\in\E_{j,p'}$ such that $\vp_e=\vp_{e'}$ and $\vp_e(K_p)\cap\vp_{e'}(K_{p'})\neq\varnothing$. This means that there is a solid edge from $(i,j)$ to $(p,p')$ and $K_p\cap K_{p'}\neq\varnothing$ (since $\vp_e=\vp_{e'}$). Moreover, there are no terminated walks starting from $(p,p')$. By the induction hypothesis, we can find a solid walk in the intersection graph which starts from $(p,p')$ and has length $k$. Splicing the previous solid edge from $(i,j)$ to $(p,p')$, we obtain a solid walk that starts from $(i,j)$ and has length $k+1$. This completes the induction.
\end{proof}

\begin{lemma}\label{lem:mainthm1_2}
    Let $(i_0,j_0)\to(i_1,j_1)\to\cdots\to(i_m,j_m)$ be a solid walk of length $m$ in the intersection graph.
    \begin{enumerate}
        \item If $m\geq\frac{n^2-n}{2}$ then $K_{i_0}\cap K_{j_0}\neq\varnothing$;
        \item If $m\geq \frac{n^2-n}{2}+1$ then there is an infinite solid walk in the intersection graph which starts from $(i_0,j_0)$.
    \end{enumerate}
\end{lemma}
\begin{proof}
    Suppose $m\geq \frac{n^2-n}{2}$. Note that there is no edge in the intersection graph which has one of $(1,1),\ldots,(n,n)$ as its initial vertex. So $i_k\neq j_k$ for $1\leq k\leq m-1$. By Lemma~\ref{lem:blackedge}, $K_{i_0}\cap K_{j_0}$ contains a scaled copy of $K_{i_{m}}\cap K_{j_{m}}$. In particular, if $i_m=j_m$ then $K_{i_0}\cap K_{j_0}\neq\varnothing$. If $i_m\neq j_m$ then, since $|\{\{i,j\}: i\neq j\}|=\frac{n^2-n}{2}\leq m<m+1$, we can find $0\leq p<q\leq m$ with $\{i_p,j_p\}=\{i_q,j_q\}$.

    If $(i_q,j_q)=(i_p,j_p)$ then $(i_p,j_p)\to(i_{p+1},j_{p+1})\to\cdots\to(i_q,j_q)$ is a solid walk from $(i_p,j_p)$ to itself. If $(i_q,j_q)=(j_p,i_p)$ then we have by the ``symmetric property'' of the intersection graph (recall the remark after Definition~\ref{de:intersectgraph}) that
    \[
        (i_q,j_q) =(j_p,i_p) \to (j_{p+1},i_{p+1}) \to\cdots\to (j_{q-1},i_{q-1}) \to (j_q,i_q)=(i_p,j_p)
    \]
    is a solid walk. Therefore,
    \[
        (i_p,j_p)\to\cdots \to(i_q,j_q)\to (j_{p+1},i_{p+1}) \to\cdots\to (j_{q-1},i_{q-1}) \to (j_q,i_q)
    \]
    is a solid walk from $(i_p,j_p)$ to itself. So in both cases, we can easily obtain an infinite solid walk that starts from $(i_p,j_p)$ and hence another one from $(i_0,j_0)$. By Proposition~\ref{prop:gds2}, $K_{i_0}\cap K_{j_0}\neq\varnothing$.

    Suppose $m\geq\frac{n^2-n}{2}+1$. Then $m-1\geq\frac{n^2-n}{2}$ and $i_{m-1}\neq j_{m-1}$. Thus one can deduce by the same argument as above that there is an infinite solid walk from $(i_0,j_0)$.
\end{proof}

\begin{proof}[Proof of Theorem~\ref{thm:main1}]
    The sufficiency follows directly from Propositions~\ref{prop:gds1} and~\ref{prop:gds2}. 
    Now we prove the necessity. Suppose $K_i\cap K_j\neq\varnothing$ but in the intersection graph, there are no terminated finite walks starting at $(i,j)$. By Lemma~\ref{lem:mainthm1_1}, there exists a solid walk which starts from $(i,j)$ and has length $\frac{n^2-n}{2}+1$. Then Lemma~\ref{lem:mainthm1_2} immediately gives us an infinite solid walk starting from $(i,j)$.
\end{proof}

%    The necessity is deduced by contradiction. Suppose %$K_i\cap K_j\neq\varnothing$ but in the intersection %graph, there are neither terminated finite walks nor %infinite solid walks starting at $(i,j)$. By %Lemma~\ref{lem:mainthm1_1}, there exists a solid walk %which starts from $(i,j)$ and has length %$\frac{n^2-n}{2}+1$. But then Lemma~\ref{lem:mainthm1_2} %immediately gives us an infinite solid walk starting from %$(i,j)$. This is a contradiction.

%----------------------------------------------------------
\subsection{An induction process to determine terminated edges}

Recall our previous remark that the drawing of the intersection graph for attractors in $\R^d$ depends on the solution of the problem in lower dimensional spaces $\R,\ldots,\R^{d-1}$. Toward this end, we need the following observation.

For any face $P$ of $[0,1]^d$ with $1\leq\dim P\leq d-1$, there is a $0$-$1$ vector $\alpha=(a_1,\ldots,a_s)$ and a sequence $1\leq m_1<\cdots<m_s \leq d$ such that \eqref{eq:P-def} holds, where $s=d-\dim P$. We denote by $h_P$ the natural projection from $\R^d$ to $\R^{\dim P}$ given by
\[
    h_P(x_1,\ldots,x_d) = (x_1,\ldots,x_{m_1-1},x_{m_1+1},\ldots,x_{m_s-1}, x_{m_s+1},\ldots,x_d).
\]

\begin{lemma}\label{lem:reducedim}
    Let $P$ be a face of $[0,1]^d$ with $1\leq \dim P\leq d-1$ and $P\cap \bigcup_{i=1}^n K_i\neq\varnothing$. Let $\Lambda_P:=\{i\in\V_n: K_i\cap P\neq\varnothing\}$. Then the tuple $(h_P(K_i\cap P))_{i\in\Lambda_P}$ can be regarded as a Cantor-type graph-directed attractor in $\R^{\dim P}$.
\end{lemma}

\begin{proof}
    Note that for every $e\in\E$ and $(x_1,\ldots,x_d)\in\R^d$,
    \begin{align*}
        h_P\circ\vp_e(x_1,\ldots,x_d) &= h_P\Big( \frac{(x_1,\ldots,x_d)+t_e}{N} \Big) \\
        &= \frac{h_P((x_1,\ldots,x_d))+h_P(t_e)}{N} = \psi_e(h_P(x_1,\ldots,x_d)),
    \end{align*}
    where $\psi_e:y\mapsto N^{-1}(y+h_P(t_e))$. Thus $h_P\circ\vp_e=\psi_e\circ h_P$ and clearly $h_P(t_e)\in\{0,1,\ldots,N-1\}^{\dim P}$. So $\{\psi_e:e\in\E\}$ is of Cantor-type. Writing $\E_{i,j}^P:=\{e\in\E_{i,j}: \vp_e(K_j)\cap P\neq\varnothing\}$ and $K_i^P:=h_P(K_i\cap P)$, we have for all $i\in\Lambda_P$ that
    \begin{equation}\label{eq:kipneweq}
        K_i^P = h_P(K_i\cap P) = h_P\Big( \bigcup_{j=1}^n\bigcup_{e\in\E_{i,j}} \vp_e(K_j)\cap P \Big) = h_P\Big( \bigcup_{j=1}^n\bigcup_{e\in\E_{i,j}^P} \vp_e(K_j)\cap P \Big).
    \end{equation}
    Note that $\E_{i,j}^P\neq\varnothing$ implies that $j\in\Lambda_P$: If there is $e\in\E_{i,j}^P$ then we have by Lemma~\ref{lem:face} that 
    \[
        \vp_e(K_j\cap P) = \vp_e(K_j) \cap P \neq\varnothing \Longrightarrow K_j\cap P\neq\varnothing.
    \]
    Combining this with Lemma~\ref{lem:face} and~\eqref{eq:kipneweq},
    \begin{align*}
        K_i^P &= h_P\Big( \bigcup_{j\in\Lambda_P}\bigcup_{e\in\E_{i,j}^P} \vp_e(K_j)\cap P \Big) \\
        %&= h_P\Big( \bigcup_{j\in\Lambda}\bigcup_{e\in\E_{i,j}^P} \vp_e(K_j)\cap\vp_e([0,1]^d)\cap P \Big) \\
        &= h_P\Big( \bigcup_{j\in\Lambda_P}\bigcup_{e\in\E_{i,j}^P} \vp_e(K_j\cap P) \Big) \\
        &= \bigcup_{j\in\Lambda_P}\bigcup_{e\in\E_{i,j}^P} h_P\circ\vp_e (K_j\cap P) \\
        &= \bigcup_{j\in\Lambda_P}\bigcup_{e\in\E_{i,j}^P} \psi_e\circ h_P(K_j\cap P) = \bigcup_{j\in\Lambda_P}\bigcup_{e\in\E_{i,j}^P} \psi_e(K^P_j).
    \end{align*}
    Thus the tuple $(K^P_i)_{i\in\Lambda_P}$ can be regarded as a Cantor-type attractor in $\R^{\dim P}$ of which the associated graph-directed system is as follows:
    \begin{enumerate}
        \item The vertex set of the directed graph is $\Lambda_P$;
        \item The edge set is $\bigcup_{i,j\in\Lambda_P} \E_{i,j}^P$ and the IFS is $\{\psi_e:e\in\bigcup_{i,j\in\Lambda_P} \E_{i,j}^P\}$.
    \end{enumerate}
    %Since $K^P_i$ is just the ``projection'' of $K_i\cap P$, this completes the proof.
\end{proof}

\begin{corollary}\label{cor:reducedim}
    Let $P_1,\ldots,P_m$ be faces of $[0,1]^d$ with $1\leq \dim P_1=\cdots=\dim P_m\leq d-1$, and let $\Lambda_{P_k}:=\{1\leq i\leq n: K_i\cap P_k\neq\varnothing\}$ for $1\leq k\leq m$. Then the tuple $(h_{P_k}(K_t\cap P_k))_{1\leq k\leq m,t\in\Lambda_{P_k}}$ can be regarded as a Cantor-type graph-directed attractor in $\R^{\dim P_1}$.
\end{corollary}

Now we are able to determine dashed terminated edges in the intersection graph. Recall that a dashed terminated edge from $(i,j)$ (where $i\neq j$) to some $(i',j')$ implies that there are $e\in\E_{i,i'}$, $e'\in\E_{j,j'}$ such that $\vp_e([0,1]^d)\cap\vp_{e'}([0,1]^d)$ is a common $s$-dimensional face of these two cubes with $0\leq s\leq d-1$ and $\vp_e(K_{i'})\cap\vp_{e'}(K_{j'})\neq\varnothing$. The last condition $\vp_e(K_{i'})\cap\vp_{e'}(K_{j'})\neq\varnothing$ can be checked by an induction process as follows.

When $d=1$, $s$ must be zero and $\vp_e([0,1])\cap\vp_{e'}([0,1])$ is just the common endpoint of these two intervals. So to determine whether $\vp_e(K_{i'})\cap\vp_{e'}(K_{j'})\neq\varnothing$, it suffices to figure out whether endpoints of $[0,1]$ are elements of $K_{i'}$ and $K_{j'}$. This task has already been done by Corollary~\ref{cor:01}. So we know exactly which dashed edges are terminated and thus how to draw the intersection graph for attractors in $\R$. Combining this with Theorem~\ref{thm:main1}, we solve the intersection problem of Cantor-type attractors in $\R$.

When $d=2$, $s$ can take values $0$ or $1$. If $s=0$ then $\vp_e([0,1]^2)\cap\vp_{e'}([0,1]^2)$ is just the common vertex of these two squares. So to determine whether $\vp_e(K_{i'})\cap\vp_{e'}(K_{j'})\neq\varnothing$, it suffices to find out whether vertices of $[0,1]^2$ are elements of $K_{i'}$ and $K_{j'}$. Again, this can be checked by Corollary~\ref{cor:01}. If $s=1$, $\vp_e(K_{i'})\cap\vp_{e'}(K_{j'})$ lives on a common $1$-dimensional face of $\vp_e([0,1]^2)$ and $\vp_{e'}([0,1]^2)$. Thus there are $1$-dimensional faces $P,Q$ of $[0,1]^2$ such that $\vp_e(K_{i'})\cap\vp_{e'}(K_{j'})=\vp_e(K_{i'}\cap P) \cap \vp_{e'}(K_{j'}\cap Q)$. Note that $Q$ is a translation of $P$ and
\[
    \vp_e(K_{i'}\cap P) \cap \vp_{e'}(K_{j'}\cap Q)\neq\varnothing \Longleftrightarrow h_P(K_{i'}\cap P) \cap h_Q(K_{j'}\cap Q) \neq\varnothing.
\] 
If $K_{i'}\cap P=\varnothing$ or $K_{j'}\cap Q=\varnothing$ (which can be checked by Corollary~\ref{cor:01}) then we are done. When both of these intersections are non-empty, it follows from Corollary~\ref{cor:reducedim} that the problem is reduced to the intersection problem of the Cantor-type attractor $(h_P(K_i\cap P),h_Q(K_j\cap Q))_{i\in\Lambda_P,j\in\Lambda_Q}$, where $\Lambda_P:=\{i:K_i\cap P\neq\varnothing\}$ and $\Lambda_Q:=\{i:K_i\cap Q\neq\varnothing\}$. This is an attractor in $\R$ because $\dim P=\dim Q=1$. Note that in the last paragraph, we have solved the intersection problem in $\R$. In particular, we can determine whether $h_P(K_{i'}\cap P)\cap h_Q(K_{j'}\cap Q)\neq\varnothing$ and thus whether $\vp_e(K_{i'})\cap\vp_{e'}(K_{j'})\neq\varnothing$. Now we know which dashed edges are terminated and how to build the intersection graph for attractors in $\R^2$. Again, combining this with Theorem~\ref{thm:main1}, we solve the intersection problem of Cantor-type attractors in $\R^2$.

Continuing in this manner, the intersection problem of Cantor-type attractors in $\R^d$ for all dimensions $d\geq 1$ is settled.

%---------------------------------------------------------------
\section{The intersection problem II: The finite-iteration approach}
%\subsection{Determine the intersection within finitely many iterations}

This section is devoted mainly to the proof of Theorem~\ref{thm:main2}. The method used in the following result is similar to Lemma~\ref{lem:mainthm1_2}.

\begin{lemma}\label{lem:01finite}
    Let $\alpha$ be a vertex of $[0,1]^d$ and let $i\in\V_n$. If there is a walk in the graph $G_{\{\alpha\}}$ which starts from $i$ and has length $n$, then $\alpha\in K_i$.
\end{lemma}
\begin{proof}
    Let $i\to i_1\to\cdots\to i_n$ be such a walk. For simplicity, set $i_0=i$. Since there are only $n$ vertices in $G_{\{\alpha\}}$, we can find $0\leq p<q\leq n$ such that $i_p=i_q$. So there is a cycle at $i_p$ and hence an infinite walk starting from $i$. By Lemma~\ref{lem:0ink1}, $\alpha\in Q_{i,k}$ for all $k$ and thus $\alpha\in K_i$.
\end{proof}

\begin{corollary}\label{cor:01finite}
    Let $\alpha$ be a vertex of $[0,1]^d$. For $i\in\V_n$, $\alpha\in K_i$ if and only if $\alpha\in Q_{i,n}$.
\end{corollary}
\begin{proof}
    Since $K_i\subset Q_{i,n}$, it suffices to show the sufficiency. Assume that $\alpha\in Q_{i,n}$. Then it follows from Lemma~\ref{lem:0ink2} that there exists a walk in the graph $G_{\{\alpha\}}$ which starts from $i$ and has length $n$. Combining this with Lemma~\ref{lem:01finite}, we see that $\alpha\in K_i$.
\end{proof}

The following are direct results of Lemma~\ref{lem:lengthkwalk}.

\begin{corollary}\label{cor:inki}
    If $i\xrightarrow{e_1}i_1\xrightarrow{e_2}i_2\to\cdots$ is an infinite walk in $G$ then
    \[
       \bigcap_{m=1}^\infty \vp_{e_1}\circ\cdots\circ\vp_{e_m}([0,1]^d) \subset K_i.
    \]
\end{corollary}
\begin{proof}
    For $m\geq 1$, $i\xrightarrow{e_1}i_1\xrightarrow{e_2}\cdots\xrightarrow{e_m}i_m$ is a walk of length $m$ starting from $i$. It then follows from Lemma~\ref{lem:lengthkwalk} that $\vp_{e_1}\circ\cdots\circ\vp_{e_m}([0,1]^d) \subset Q_{i,m}$. 
    It is not hard to see that $\{\vp_{e_1}\circ\cdots\circ\vp_{e_m}([0,1]^d)\}_{m=1}^\infty$ is a nested sequence and hence
    \[
       \bigcap_{m=1}^\infty \vp_{e_1}\circ\cdots\circ\vp_{e_m}([0,1]^d) \subset \bigcap_{m=1}^\infty Q_{i,m}=K_i.
    \]
\end{proof}

\begin{corollary}\label{cor:levelmsolidwalk}
    Let $m\geq 1$. For every pair of distinct $i,j\in\V_n$, if there is a level-$m$ cube $I\subset Q_{i,m}\cap Q_{j,m}$, then we can find a solid walk in the intersection graph which starts from $(i,j)$ and has length $m$.
\end{corollary}
\begin{proof}
    We shall prove this by induction. When $m=1$, let $w\in\E^1_i$, $w'\in\E_j^1$ be such that $\vp_w([0,1]^d)=\vp_{w'}([0,1]^d)$. So $\vp_w=\vp_{w'}$ and hence there is a solid edge in the intersection graph which starts from $(i,j)$.

    Suppose the statement holds for $1\leq m\leq k$. When $m=k+1$, let $i\xrightarrow{e_1} i_1\xrightarrow{e_2} i_2\to\cdots\to i_{k+1} \in\E_i^{k+1}$ and $j\xrightarrow{e'_1} j_1\xrightarrow{e'_2} j_2\to\cdots\to j_{k+1} \in\E_j^{k+1}$ be such that 
    \[
        I = \vp_{e_1}\circ\cdots\circ\vp_{e_{k+1}}([0,1]^d) = \vp_{e'_1}\circ\cdots\circ\vp_{e'_{k+1}}([0,1]^d).
    \]
    This implies that $\vp_{e_1}=\vp_{e'_1}$, so there is a solid edge from $(i,j)$ to $(i_1,j_1)$. Moreover, 
    \[
        \vp_{e_1}^{-1}(I) = \vp_{e_2}\circ\cdots\circ\vp_{e_{k+1}}([0,1]^d) = \vp_{e'_2}\circ\cdots\circ\vp_{e'_{k+1}}([0,1]^d) 
    \]
    is a level-$k$ cube contained in $Q_{i_1,k}\cap Q_{j_1,k}$. As a consequence, we can find by the induction hypothesis a solid walk in the intersection graph which starts from $(i_1,j_1)$ and has length $k$. Splicing the previous solid edge from $(i,j)$ to $(i_1,j_1)$, we obtain a solid walk which starts from $(i,j)$ and has length $k+1$. This completes the induction.
\end{proof}

We need another fact for technical reasons.

\begin{lemma}\label{lem:equalspandq}
    Let $\vp(x)=N^{-a}(x+t)$, $\vp_*(x)=N^{-a}(x+t_*)$, $f(x)=N^{-b}(x+w)$, $f_*(x)=N^{-b}(x+w_*)$ be contracting maps on $\R^d$ such that $a,b\in\Z^+$, $t,t_*,w,w_*\in\R^d$ and all of the four mappings send $[0,1]^d$ into itself. If $P,Q$ are nonempty subsets of $[0,1]^d$ such that $\vp(P)=\vp_*(Q)$ and $\vp\circ f(P)=\vp_*\circ f_*(Q)$, then
    \[
        \vp\circ f^m(P)=\vp_*\circ f^m_*(Q), \quad m\geq 1.
    \]
\end{lemma}
\begin{proof}
    Note that $\vp(P)=\vp_*(Q)$ implies $P=Q+t_*-t$ and $\vp\circ f(P)=\vp_*\circ f_*(Q)$ implies $P=Q+w_*+N^bt_*-w-N^bt$. In particular, we have $t_*-t=w_*+N^bt_*-w-N^bt$ so
    \begin{equation}\label{eq:ttstar}
        (N^b-1)(t-t_*)=w_*-w.
    \end{equation}
    Since $P=Q+t_*-t$, we have for all $m\geq 1$ that 
    \begin{align*}
        N^a\vp\circ f^m(P) &=  \frac{1}{N^{mb}}P+\sum_{k=1}^m \frac{1}{N^{kb}}w+t  \\
        &= \frac{1}{N^{mb}}(Q+t_*-t)+\sum_{k=1}^m \frac{1}{N^{kb}}w+t  \\
        &=\Big( \frac{1}{N^{mb}}Q+\sum_{k=1}^m \frac{1}{N^{kb}}w_*+t_* \Big) + \Big( 1-\frac{1}{N^{mb}} \Big)(t-t_*) 
       + \sum_{k=1}^m \frac{1}{N^{kb}}(w-w_*).
    \end{align*}
    The term in the first bracket equals $N^a\varphi_*\circ f_*^m(Q)$, so it suffices to show that the sum of the last two terms vanishes. But this is straightforward: By~\eqref{eq:ttstar},
    \begin{align*}
        &\Big( 1-\frac{1}{N^{mb}} \Big)(t-t_*)+ \sum_{k=1}^m \frac{1}{N^{kb}}(w-w_*) \\
        &\qquad = \Big( 1-\frac{1}{N^{mb}} \Big)(t-t_*) - \frac{N^{-b}(1-N^{-mb})}{1-N^{-b}}(N^b-1)(t-t_*) = 0.
    \end{align*} 
\end{proof}

Now we have all the ingredients needed to establish Theorem~\ref{thm:main2}.

\begin{proof}[Proof of Theorem~\ref{thm:main2}]
    Again, we only need to show the sufficiency. Suppose $Q_{i,c(n,d)} \cap Q_{j,c(n,d)}\neq\varnothing$ and arbitrarily pick a point $x$ in this intersection. By Lemma~\ref{lem:lengthkwalk}, we can find
    \[
        i\xrightarrow{e_1} i_1\xrightarrow{e_2} i_2\to\cdots\to i_{c(n,d)} \in\E_i^{c(n,d)},\,\, j\xrightarrow{e'_1} j_1\xrightarrow{e'_2} j_2\to\cdots\to j_{c(n,d)} \in\E_j^{c(n,d)}
    \]
    such that
    \[
        x\in \vp_{e_1}\circ\cdots\circ\vp_{e_{c(n,d)}}([0,1]^d)\cap\vp_{e'_1}\circ\cdots\circ\vp_{e'_{c(n,d)}}([0,1]^d) \subset Q_{i,c(n,d)}\cap Q_{j,c(n,d)}.
    \]
    Moreover, we have for $1\leq k\leq c(n,d)$ that
    \[
        \vp_{e_1}\circ\cdots\circ\vp_{e_k}([0,1]^d)\subset Q_{i,k},\,\, \vp_{e'_1}\circ\cdots\circ\vp_{e'_k}([0,1]^d)\subset Q_{j,k},
    \]
    and $x$ is a common point of the above two level-$k$ cubes. Define
    \[
        s_k:=\dim \big( \vp_{e_1}\circ\cdots\circ\vp_{e_k}([0,1]^d)\cap \vp_{e'_1}\circ\cdots\circ\vp_{e'_k}([0,1]^d) \big), \quad 1\leq k\leq c(n,d).
    \]
    Clearly, $s_1\geq s_2\geq\cdots\geq s_{c(n,d)}\geq 0$. We will discuss the following three cases separately.

    \textbf{Case I}: $|\{k:s_k=0\}|\geq n+1$. So $s_{c(n,d)-n}=\cdots=s_{c(n,d)}=0$. This implies that
    \[
        \{x\} = \vp_{e_1}\circ\cdots\circ\vp_{e_k}([0,1]^d)\cap \vp_{e'_1}\circ\cdots\circ\vp_{e'_k}([0,1]^d), \quad c(n,d)-n\leq k\leq c(n,d),
    \]
    i.e., $x$ is the common vertex of these cubes. In particular, there are vertices $\alpha,\beta$ of $[0,1]^d$ such that
    \[
        x = \vp_{e_1}\circ\cdots\circ\vp_{e_{c(n,d)-n}}(\alpha) = \vp_{e'_1}\circ\cdots\circ\vp_{e'_{c(n,d)-n}}(\beta).
    \]
    By Lemma~\ref{lem:lengthkwalk},
    \[
        \vp_{e_{c(n,d)-n+1}}\circ\cdots\circ\vp_{e_{c(n,d)}}([0,1]^d) \subset Q_{i_{c(n,d)-n},n},
    \]
    implying that
    \[
        \vp_{e_1}\circ\cdots\circ\vp_{e_{c(n,d)-n}}(\alpha) = x \in \vp_{e_1}\circ\cdots\circ\vp_{e_{c(n,d)-n}}(Q_{i_{c(n,d)-n},n}).
    \]
    Thus $\alpha\in Q_{i_{c(n,d)-n},n}$ since $\vp_{e_1}\circ\cdots\circ\vp_{e_{c(n,d)-n}}$ is injective. Similarly, $\beta\in Q_{j_{c(n,d)-n},n}$.
    By Corollary~\ref{cor:01finite}, we have $\alpha\in K_{i_{c(n,d)-n}}$ and $\beta\in K_{j_{c(n,d)-n}}$. This further implies that
    \begin{align*}
        \{x\} &= \{\vp_{e_1}\circ\cdots\circ\vp_{e_{c(n,d)-n}}(\alpha)\} \cap \{\vp_{e'_1}\circ\cdots\circ\vp_{e'_{c(n,d)-n}}(\beta)\} \\
        &\subset \vp_{e_1}\circ\cdots\circ\vp_{e_{c(n,d)-n}}(K_{i_{c(n,d)-n}}) \cap \vp_{e'_1}\circ\cdots\circ\vp_{e'_{c(n,d)-n}}(K_{j_{c(n,d)-n}})  \\
        &\subset K_i\cap K_j,
    \end{align*}
    where the last step follows again from Lemma~\ref{lem:lengthkwalk}. In particular, $K_i\cap K_j\neq\varnothing$.

    \textbf{Case II}: $|\{k:s_k=d\}|\geq \frac{n^2-n}{2}$. So $s_1=s_2=\cdots=s_{(n^2-n)/2}=d$. In this case, we have
    \[
        \vp_{e_1}\circ\cdots\circ\vp_{e_{(n^2-n)/2}}([0,1]^d) = \vp_{e'_1}\circ\cdots\circ\vp_{e'_{(n^2-n)/2}}([0,1]^d),
    \]
    which is a level-$(\frac{n^2-n}{2})$ cube contained in $Q_{i,(n^2-n)/2}\cap Q_{j,(n^2-n)/2}$. By Corollary~\ref{cor:levelmsolidwalk}, there is a solid walk in the intersection graph which starts from $(i,j)$ and has length $\frac{n^2-n}{2}$. It then follows from Lemma~\ref{lem:mainthm1_2} that $K_i\cap K_j\neq\varnothing$.

    \textbf{Case III}: $d-1\geq s_{(n^2-n)/2}\geq\cdots\geq s_{c(n,d)-n}\geq 1$. For simplicity, we temporarily write $c_n:=\frac{n^2-n}{2}$. Since 
    \[
        c(n,d)-n-c_n+1= (d-1)n^2+d,
    \]
    there is some $1\leq s\leq d-1$ such that $|\{c_n\leq k\leq c(n,d)-n:s_k=s\}|\geq n^2+2$. Without loss of generality, assume that $s_{c_n}=\cdots=s_{c_n+n^2+1}=s$. Since $|\{(a,b):a,b\in\V_n\}|=n^2$, there are $c_n+1\leq p<q\leq c_n+n^2+1$ such that $(i_p,j_p)=(i_q,j_q)$. Then
    \begin{equation*}
    	\begin{gathered}
    		i\xrightarrow{e_1}\cdots\xrightarrow{e_p} i_p\xrightarrow{e_{p+1}}\cdots\xrightarrow{e_q} i_q=i_p\xrightarrow{e_{p+1}}\cdots\xrightarrow{e_q} i_q \to\cdots, \text{and}\\
    		j\xrightarrow{e'_1}\cdots\xrightarrow{e'_p} j_p\xrightarrow{e'_{p+1}}\cdots\xrightarrow{e'_q} j_q=j_p\xrightarrow{e'_{p+1}}\cdots\xrightarrow{e'_q} j_q \to\cdots
    	\end{gathered}
    \end{equation*}
%    \[
%        i\xrightarrow{e_1}\cdots\xrightarrow{e_p} i_p\xrightarrow{e_{p+1}}\cdots\xrightarrow{e_q} i_q=i_p\xrightarrow{e_{p+1}}\cdots\xrightarrow{e_q} i_q \to\cdots
%    \]
%    and
%    \[
%        j\xrightarrow{e'_1}\cdots\xrightarrow{e'_p} j_p\xrightarrow{e'_{p+1}}\cdots\xrightarrow{e'_q} j_q=j_p\xrightarrow{e'_{p+1}}\cdots\xrightarrow{e'_q} j_q \to\cdots
%    \]
    are infinite walks in $G$. Denoting $\psi:=\vp_{e_{p+1}}\circ\cdots\circ\vp_{e_{q}}$ and $\psi_*:=\vp_{e'_{p+1}}\circ\cdots\circ\vp_{e'_{q}}$, we have by Corollary~\ref{cor:inki} that
    \begin{equation}\label{eq:twoptsaresame}
       \bigcap_{m=1}^\infty \vp_{e_1}\circ\cdots\circ\vp_{e_{p}}\circ\psi^m([0,1]^d)\subset K_i, \quad \bigcap_{m=1}^\infty  \vp_{e'_1}\circ\cdots\circ\vp_{e'_{p}}\circ\psi_*^m([0,1]^d)\subset K_j. 
    \end{equation}
    Recall that
    \begin{align*}
        &\dim \big( \vp_{e_1}\circ\cdots\circ\vp_{e_{p}}([0,1]^d)\cap \vp_{e'_1}\circ\cdots\circ\vp_{e'_{p}}([0,1]^d) \big) \\
        &\qquad = s =\dim \big( \vp_{e_1}\circ\cdots\circ\vp_{e_{p}}\circ\psi([0,1]^d)\cap \vp_{e'_1}\circ\cdots\circ\vp_{e'_{p}}\circ\psi_*([0,1]^d) \big).
    \end{align*}
    Let 
    $
      R=\vp_{e_1}\circ\cdots\circ\vp_{e_{p}}([0,1]^d)\cap \vp_{e'_1}\circ\cdots\circ\vp_{e'_{p}}([0,1]^d),
    $
    and define
    \[
      P=\big( \vp_{e_1}\circ\cdots\circ\vp_{e_{p}}\big)^{-1}(R), \quad 
      Q=\big(\vp_{e'_1}\circ\cdots\circ\vp_{e'_{p}} \big)^{-1}(R).
    \]
    Then it is straightforward to see that $P$ and $Q$ are $s$-dimensional faces of $[0,1]^d$.      Combining with the monotonicity, we have
%    Combining with the monotonicity, there are %$s$-dimensional faces $P,Q$ of $[0,1]^d$ such that
    \[
        \left\{\begin{array}{l}
        	\vp_{e_1}\circ\cdots\circ\vp_{e_{p}}(P)=\vp_{e'_1}\circ\cdots\circ\vp_{e'_{p}}(Q), \\
        	\vp_{e_1}\circ\cdots\circ\vp_{e_{p}}\circ\psi(P) = \vp_{e'_1}\circ\cdots\circ\vp_{e'_{p}}\circ\psi_*(Q).
        \end{array}\right.
    \]
    From Lemma~\ref{lem:equalspandq},
    \[
       \bigcap_{m=1}^\infty \vp_{e_1}\circ\cdots\circ\vp_{e_{p}}\circ\psi^m(P) = \bigcap_{m=1}^\infty \vp_{e'_1}\circ\cdots\circ\vp_{e'_{p}}\circ\psi_*^m(Q), 
    \]
    which is a singleton contained in $K_i\cap K_j$ (recall~\eqref{eq:twoptsaresame}). In particular, $K_i\cap K_j\neq\varnothing$.
\end{proof}

\begin{remark}
    The constant $c(n,d)$ in Theorem~\ref{thm:main2} may not be optimal in general. In fact, we would not be too surprised if one could show that $c(n,d)$ can be taken not greater than $d\cdot\frac{(n^2-n)}{2}+(d-1)+n$, which seems to require careful analysis of the structure of the graph-directed system. When $n=1$ (which is just the self-similar case), this is essentially proved in~\cite{DLRWX21} in a slightly different language. 
\end{remark}

%------------------------------------------
\section{Connectedness of sponge-like sets}

Let $F=F(d,N,\I)$ be any fixed sponge-like set in $\R^d$. By Hata's criterion, to determine whether $F$ is connected, it suffices to draw the associated Hata graph. This requires our knowledge of the emptiness of $\vp_i(F) \cap \vp_{j}(F)$ for every pair of $i,j\in\I$. Recall that $\mathcal{O}_d$ denotes the group of symmetries of the unit $d$-cube $[0,1]^d$.

\begin{lemma}\label{lem:cantortypeatt}
    The tuple $(O(F))_{O\in\mathcal{O}_d}$ forms a Cantor-type graph-directed attractor in $\R^d$.
\end{lemma}
\begin{proof}
    It is well known that the group of symmetries of $[-\frac{1}{2},\frac{1}{2}]^d$ is the collection $\mathcal{M}_d$ of $d\times d$ matrices with entries only $0,\pm 1$ and with exactly one non-zero entry in each row and column. Furthermore, $|\mathcal{M}_d|=d!2^d$ (so $|\mathcal{O}_d|=d!2^d$). Writing $\mathbf{a}_d$ to be the $d$-tuple $(\frac{1}{2},\ldots,\frac{1}{2})$, each $O\in\mathcal{O}_d$ corresponds to a matrix $A_O\in\mathcal{M}_d$ such that
    \begin{equation}\label{eq:expressofo}
        O(x) = A_O(x-\mathbf{a}_d)+\mathbf{a}_d.
    \end{equation}
    This is just setting up a conjugacy between mappings with the origin shifted to the center of $[0,1]^d$, which brings us technical convenience when referring to symmetries of $[0,1]^d$.

    Since $F=\bigcup_{i\in\I}\vp_i(F)$, where $\varphi_i(x)=N^{-1}(O_i(x)+i)$, we have for each $O\in\mathcal{O}_d$ that
    \[
        O(F) = O\Big( \bigcup_{i\in\I}\vp_i(F) \Big) = O\Big( \bigcup_{i\in\I}\frac{O_i(F)+i}{N} \Big).
    \]
    It then follows from~\eqref{eq:expressofo} that
    \begin{align*}
        O(F) &= \bigcup_{i\in\I} O\Big( \frac{A_{O_i}(F-\mathbf{a}_d)+\mathbf{a}_d+i}{N} \Big) \\
        &= \bigcup_{i\in\I} A_O\Big( \frac{A_{O_i}(F-\mathbf{a}_d)+\mathbf{a}_d+i}{N}-\mathbf{a}_d \Big)+\mathbf{a}_d. 
    \end{align*}
    Rewriting this gives us 
    \begin{align*}
        O(F) &= \bigcup_{i\in\I}\frac{A_OA_{O_i}(F-\mathbf{a}_d)+\mathbf{a}_d}{N}+A_O\Big( \frac{\mathbf{a}_d+i}{N}-\mathbf{a}_d \Big) + \frac{N-1}{N}\mathbf{a}_d \\
        &= \bigcup_{i\in\I}\frac{O\circ O_i(F)}{N}+\frac{1}{N}A_O\Big( i-(N-1)\mathbf{a}_d \Big) + \frac{N-1}{N}\mathbf{a}_d. 
    \end{align*}
    Since entries of $i$ only take value in $\{0,1,\ldots,N-1\}$, entries of the second term above only take value in $N^{-1}\{-\tfrac{N-1}{2},-\tfrac{N-1}{2}+1,\ldots,\tfrac{N-1}{2}\}$. So the sum of the last two terms only takes value in $\{0,\frac{1}{N},\ldots,\frac{N-1}{N}\}^d$. Combining with the fact that $O\circ O_i\in\mathcal{O}_d$, we see that $(O(F))_{O\in\mathcal{O}_d}$ forms a Cantor-type graph-directed attractor.

    More precisely, enumerating $\mathcal{O}_d=\{\widetilde{O}_1,\ldots,\widetilde{O}_{d!2^d}\}$, we have
    \begin{equation}\label{eq:okfrepresentation}
        \widetilde{O}_k(F) = \bigcup_{i\in\I} \frac{\widetilde{O}_k\circ O_i(F)+t_{k,i}}{N} = \bigcup_{i\in\I} \frac{\widetilde{O}_{\Omega(k,i)}(F)+t_{k,i}}{N}, \quad 1\leq k\leq d!2^d,
    \end{equation}
    where $\Omega(k,i)$ is such that $\widetilde{O}_{\Omega(k,i)}=\widetilde{O}_k\circ O_i$ and
    \[
        t_{k,i}:= A_{\widetilde{O}_k}\Big(  i-(N-1)\mathbf{a}_d \Big) + (N-1)\mathbf{a}_d. 
    \]
    So the graph-directed structure associated with $(\widetilde{O}_1(F),\ldots,\widetilde{O}_{d!2^d}(F))$ is as follows.
    \begin{enumerate}
        \item The vertex set can be labelled as $\{1,2,\ldots,d!2^d\}$;
        \item For every $1\leq k\leq d!2^d$ and every $i\in\I$, we add an edge from $k$ to $\Omega(k,i)$ and assign the homothety corresponding to this edge to be $x\mapsto N^{-1}(x+t_{k,i})$.
    \end{enumerate}
\end{proof}

For convenience, we keep enumerating $\mathcal{O}_d$ as $\{\widetilde{O}_1,\widetilde{O}_2,\ldots,\widetilde{O}_{d!2^d}\}$ in the rest of this section. It is noteworthy that in the above proof, we actually show that
\begin{equation}\label{eq:ocircvpi}
    \widetilde{O}_k\circ \vp_i(x) = \frac{\widetilde{O}_{\Omega(k,i)}(x)+t_{k,i}}{N}, \quad i\in\I, x\in\R^d.
\end{equation}
This formula will be used later. It turns out that $\{\widetilde{O}_k(F_t)\}_{t=0}^\infty$ is just the sequence of geometric approximations of $\widetilde{O}_k(F)$, where $F_t$ is as in~\eqref{eq:notationoffk}.

\begin{lemma}\label{lem:approofof}
    Let $1\leq k\leq d!2^d$ and let $t\geq 0$. Then $\widetilde{O}_k(F_t)$ is the level-$t$ approximation of $\widetilde{O}_k(F)$.
\end{lemma}
\begin{proof}
    By~\eqref{eq:ocircvpi}, we have for all $t\geq 0$
    \[
        \widetilde{O}_k(F_{t+1}) = \widetilde{O}_k\Big( \bigcup_{i\in\I} \vp_i(F_t) \Big) = \bigcup_{i\in\I} \widetilde{O}_k\circ\vp_i(F_t) = \bigcup_{i\in\I} \frac{\widetilde{O}_{\Omega(k,i)}(F_t)+t_{k,i}}{N}.
    \]
    This implies that $\{\widetilde{O}_k(F_t)\}_{t=0}^\infty$ satisfies the same recursive relation as does the level-$t$ approximations of $\widetilde{O}_k(F)$ (see~\eqref{eq:okfrepresentation}). Since $\widetilde{O}_k(F_0)=\widetilde{O}_k([0,1]^d)=[0,1]^d$ is just the level-$0$ approximation of $\widetilde{O}_k(F)$, these two sequences must be identical.
\end{proof}

%------------------------------
\subsection{Hata graphs and the proof of Theorem~\ref{thm:main3}}

For distinct $i,j\in\I$, a direct observation is that $\vp_i(F) \cap \vp_{j}(F)=\varnothing$ whenever $\vp_i([0,1]^d)$ and $\vp_{j}([0,1]^d)$ are not adjacent. So it suffices to consider cases when these two cubes intersect with each other (equivalently, $i-j$ is a $0$-$\pm 1$ vector). Note that
\begin{equation}\label{eq:vpicapvpj}
    \vp_i(F) \cap \vp_{j}(F) = \frac{O_i(F)+i}{N} \cap \frac{O_{j}(F)+j}{N} = \frac{O_{j}(F)\cap (O_i(F)+i-j)}{N}+\frac{j}{N}.
\end{equation}
So $\vp_i(F) \cap \vp_{j}(F)\neq\varnothing$ if and only if $O_{j}(F)\cap (O_i(F)+i-j)\neq\varnothing$.

\textbf{Case I}: $i-j\in\{(a_1,\ldots,a_d):|a_t|=1 \text{ for all $1\leq t\leq d$}\}$. In this case,
\[
    O_{j}([0,1]^d)\cap (O_{i}([0,1]^d)+i-j) = [0,1]^d \cap ([0,1]^d+i-j)
\]
is a vertex, say $\alpha$, of the unit cube. So to see whether $\vp_i(F) \cap \vp_{i'}(F)$ is empty, it suffices to check  whether $\alpha$ is a common point of $O_j(F)$ and $O_i(F)+i-j$. Equivalently, it suffices to check whether $\alpha\in O_j(F)$ and whether $\alpha+j-i\in O_i(F)$. Note that $\alpha+j-i$ is also a vertex of $[0,1]^d$. Since $(O(F))_{O\in\mathcal{O}_d}$ forms a graph-directed attractor of Cantor-type (Lemma~\ref{lem:cantortypeatt}), this can be achieved by drawing the corresponding graphs introduced in Section 2.1.

\textbf{Case II}: $i-j\in\{(a_1,\ldots,a_d):|a_t|\leq 1 \text{ for $1\leq t\leq d$ with } 0<\Sigma_{t}|a_t|<d\}$. In this case,
\[
    O_{j}([0,1]^d)\cap (O_{i}([0,1]^d)+i-j) = [0,1]^d\cap ([0,1]^d+i-j)
\]
is a lower dimensional face $P$ of the unit cube. So to see whether $\vp_i(F) \cap \vp_{j}(F)$ is empty, it suffices to check that whether
\[
    (O_{j}(F)\cap P) \cap ((O_{i}(F)+i-j)\cap P)=\varnothing.
\]
Since $P$ and $P+j-i$ (which is also a face of $[0,1]^d$) are parallel, this is equivalent to verifying whether
\begin{align*}
    h_P\big(O_{j}(F)\cap P\big) &\cap h_P\big((O_{i}(F)+i-j)\cap P\big) \\
    &=
    h_P\big(O_{j}(F)\cap P\big) \cap h_{P+j-i}\big(O_{i}(F)\cap(P+j-i)\big)=\varnothing.
\end{align*}
Corollary~\ref{cor:reducedim} (and Corollary~\ref{cor:01} if necessary) implies that this can be reduced to the intersection problem in $\R^{\dim P}$ and we can solve this using Theorems~\ref{thm:main1} or~\ref{thm:main2}.

Similar to the graph-directed setting, we can determine the connectedness of $F$ within finitely many iterations.

\begin{proposition}\label{prop:aidofthmmain3}
    Let $1\leq i,j\leq d!2^d$ and let $\alpha\neq\mathbf{0}$ be any $0$-$\pm 1$ vector. If $\widetilde{O}_i(F_{C(d)-1}) \cap (\widetilde{O}_j(F_{C(d)-1})+\alpha)\neq\varnothing$ then $\widetilde{O}_i(F) \cap (\widetilde{O}_j(F)+\alpha)\neq\varnothing$, where $C(d)$ is as in Theorem~\ref{thm:main3}.
\end{proposition}
\begin{proof}
    Without loss of generality, assume that $\alpha$ is a $0$-$1$ vector (other cases can be similarly discussed). By Lemma~\ref{lem:cantortypeatt}, the tuple $(\widetilde{O}_k(F))_{k=1}^{d!2^d}$ is a Cantor-type graph-directed attractor in $\R^d$. Let us add two more vertices into the associated graph-directed system, namely $v_1$ and $v_2$. We add an edge from $v_1$ to $i$ with the corresponding similarity $x\mapsto N^{-1}x$, and another edge from $v_2$ to $j$ with the corresponding similarity $x\mapsto N^{-1}(x+\alpha)$. Letting $K_1:=N^{-1}\widetilde{O}_i(F)$ and $K_2:=N^{-1}(\widetilde{O}_j(F)+\alpha)$, it is not hard to see that $(\widetilde{O}_1(F),\ldots,\widetilde{O}_{d!2^d}(F),K_1,K_2)$ is the attractor associated with the new graph-directed system.

    Lemma~\ref{lem:approofof} tells us that $N^{-1}\widetilde{O}_i(F_{C(d)-1})$ and $N^{-1}(\widetilde{O}_j(F_{C(d)-1})+\alpha)$  are the level-$(C(d))$ approximations of $K_1$ and $K_2$, respectively. So by the assumption of this proposition, they have a non-empty intersection. It then follows immediately from the definition of $C(d)$ and Theorem~\ref{thm:main2} that $K_1\cap K_2\neq\varnothing$. Equivalently, $\widetilde{O}_i(F) \cap (\widetilde{O}_j(F)+\alpha)\neq\varnothing$.
\end{proof}

\begin{proof}[Proof of Theorem~\ref{thm:main3}]
    Since $F_n\supset F$ for every $n\geq 1$, it suffices to show the sufficiency. Note that $F_{C(d)}=\bigcup_{i\in\I} \vp_i(F_{C(d)-1})$ is a finite union of compact sets. Then it follows from the connectedness of $F_{C(d)}$ that for each pair of $i,j\in\I$, there exist $i_1,\ldots,i_m\in\I$ such that $i_1=i$, $i_m=j$, and $\vp_{i_k}(F_{C(d)-1})\cap\vp_{i_{k+1}}(F_{C(d)-1})\neq\varnothing$ for $1\leq k\leq m-1$.

    Since $F_n\subset[0,1]^d$ for all $n$, every coordinate of $i_k-i_{k+1}$ has absolute value not greater than $1$ (so it is a $0$-$\pm 1$ vector). With $F$ replaced by $F_{C(d)-1}$ in~\eqref{eq:vpicapvpj}, we see that $\vp_{i_k}(F_{C(d)-1})\cap\vp_{i_{k+1}}(F_{C(d)-1})\neq\varnothing$ is a scaled copy of $O_{i_{k+1}}(F_{C(d)-1})\cap (O_{i_k}(F_{C(d)-1})+i_k-i_{k+1})$. It follows immediately from Proposition~\ref{prop:aidofthmmain3} that
    \[
        O_{i_{k+1}}(F)\cap (O_{i_k}(F)+i_k-i_{k+1}) \neq\varnothing,
    \]
    which in turn implies that $\vp_{i_k}(F)\cap\vp_{i_{k+1}}(F)\neq\varnothing$. Since this holds for all $k$, we see by Hata's criterion that $F$ is connected.
\end{proof}

%-------------------------------
\subsection{Improvements on the constant $C(d)$}

In some circumstances, it is possible to determine the connectedness of $F$ more quickly than applying Theorem~\ref{thm:main3}. For example, the following result indicates that when $O_i\equiv\id$ for all $i\in\mathcal{I}$, it suffices to iterate $d+1$ times.

\begin{proposition}[\cite{DLRWX21}]\label{prop:gscconnected}
    Let $F$ be a Sierpi\'nski sponge in $\R^d$. Then $F$ is connected if and only if $F_{d+1}$ is connected.
\end{proposition}

This result can be further extended as follows.

\begin{proposition}
    Assume that there is some $O\in\mathcal{O}_d$ such that $O_i=O$ for all $i\in\I$. Denote $m$ to be the order of $O$ (i.e., the smallest integer $m\in\Z^+$ such that $O^m=\id$). Then $F$ is connected if and only if $F_{(d+1)m}$ is connected.
\end{proposition}
\begin{proof}
    For $k\geq 1$, define $\theta_k$ to be the map on $\{0,1,\ldots,N-1\}^d$ given by
    \[
       \theta_k: i\mapsto A_{O^k}(i-(N-1)\mathbf{a}_d)+(N-1)\mathbf{a}_d,
    \]
    where $A_{O^k}$ and $\mathbf{a}_d$ are as in the proof of Lemma~\ref{lem:cantortypeatt}. It then follows from~\eqref{eq:ocircvpi} that
    \[
        O^k(F) =  \bigcup_{i\in\I}O^k\circ\vp_i(F) = \bigcup_{i\in\I} \frac{O^{k+1}(F)+\theta_k(i)}{N}, \quad k\geq 1.
    \]
    Applying this repeatedly, we see that
    \begin{align*}
        F &= \bigcup_{i_1\in\I}\frac{O(F)+i_1}{N} \\
        %&= \bigcup_{i_1\in\I}\frac{\bigcup_{i_2\in\I} \frac{O^2(F)+\theta_1(i_2)}{N}+i_1}{N} \\
        %&= \bigcup_{i_1\in\I}\bigcup_{i_2\in\I} \frac{O^2(F)+\theta_1(i_2)+Ni_1}{N^2} \\
        %&=\cdots \\
        &= \bigcup_{i_1\in\I}\cdots\bigcup_{i_m\in\I} \frac{O^m(F)+\theta_{m-1}(i_m)+N\theta_{m-2}(i_{m-1})+\cdots+N^{m-1}i_1}{N^m} \\
        &= \bigcup_{j\in\mathcal{J}} \frac{F+j}{N^m},
    \end{align*}
    where $\mathcal{J}:=\theta_{m-1}(\I)+N\theta_{m-2}(\I)+\cdots+N^{m-1}\I$. This means that $F$ is a Sierpi\'nski sponge associated with the IFS $\{N^{-m}(x+j):j\in\mathcal{J}\}$. Note that (similarly as above)
    \[
        F_{m} = \bigcup_{i_1\in\I}\frac{O(F_{m-1})+i_1}{N} = \cdots = \bigcup_{j\in\mathcal{J}} \frac{F_0+j}{N} =  \bigcup_{j\in\mathcal{J}} \frac{[0,1]^d+j}{N^m}.
    \]
    Then the desired equivalence follows directly from Proposition~\ref{prop:gscconnected}.
\end{proof}

\begin{proposition}
    Assume that $\{0,N-1\}^{d}\subset\I$. Then $F$ is connected if and only if $F_1$ is connected.
\end{proposition}
\begin{proof}
    The ``$\Longrightarrow$'' part is immediate. For the ``$\Longleftarrow$'' part, we will first show that every vertex of $[0,1]^d$ is an element of $F$.

    To this end, constructing the graphs introduced in Section 2.1 of course works, but we will use a simpler method here. We will show by induction that every vertex of $[0,1]^d$ is an element of $F_n$ for all $n\geq 0$ and hence belongs to $F$. When $n=0$, $F_0=[0,1]^d$ so there is nothing to prove. Suppose that $F_n$ contains all vertices of $[0,1]^d$ for $0\leq n\leq m$. Fix any vertex $\alpha=(a_1,\ldots,a_d)$ of $[0,1]^d$. Note that
    \[
      \alpha = \frac{\alpha+(N-1)\alpha}{N}=\vp_{(N-1)\alpha}(O_{(N-1)\alpha}^{-1}\alpha),
    \]
    and by inductive assumption, $O_{(N-1)\alpha}^{-1}\alpha\in F_m$. Thus
    \[
        \alpha =\vp_{(N-1)\alpha}(O_{(N-1)\alpha}^{-1}\alpha) \in \vp_{(N-1)\alpha}(F_m) \subset \bigcup_{i\in\I}\vp_i(F_m) = F_{m+1},
    \]
    which completes the induction (note that $\alpha$ might not be the fixed point of $\varphi_{(N-1)\alpha}$).
    
%    For any $i=(i_1,\ldots,i_d)\in\I$, note that
%    \[
%        \alpha\in \vp_i([0,1]^d) \Longrightarrow (a_1,\ldots,a_d)\in \Big[ \frac{i_1}{N},\frac{i_1+1}{N} \Big] \times\cdots\times \Big[ %\frac{i_d}{N},\frac{i_d+1}{N} \Big].
%    \]
%    Discussing analogously as in the proof of Corollary~\ref{lem:fixpt}, if $a_k=0$ then $i_k=0$ and if $a_k=1$ then $i_k=N-1$. We conclude that %$i=(N-1)\alpha$. As a result,
%    \[
%        \alpha = \frac{\alpha+(N-1)\alpha}{N}=\vp_{(N-1)\alpha}(O_{(N-1)\alpha}^{-1}\alpha) \in \vp_{(N-1)\alpha}(F_m) \subset \bigcup_{j\in\I}\vp_j(F_m) = %F_{m+1},
%    \]

    Since $F_1=\bigcup_{i\in\I}\vp_i([0,1]^d)$ is connected, for any given pair of digits $i,j\in\I$, there is a sequence $\{i_k\}_{k=1}^m\subset\I$ such that $i_1=i$, $i_m=j$ and
    \[
        \vp_{i_k}([0,1]^d) \cap \vp_{i_{k+1}}([0,1]^d) \neq\varnothing, \quad 1\leq k\leq m-1.
    \]
    This means that $\vp_{i_k}([0,1]^d)$ and $\vp_{i_{k+1}}([0,1]^d)$ are adjacent cubes. Combining this with the fact that $F$ contains every vertex of $[0,1]^d$, $\vp_{i_k}(F)\cap\vp_{i_{k+1}}(F)\neq\varnothing$ for $1\leq k\leq m-1$. So $F$ is connected (again by Hata's criterion).
\end{proof}

%-------------------------
\section{Further remarks}

There are also general settings in which the previous approach to the intersection problem works. The key requirement is that given any $k$, the intersection of any pair of level-$k$ cubes must be their largest common face. Our strategy remains applicable as long as this holds.
For example, the assumption that all similarities in the IFS have the same contraction ratio $N^{-1}$ can be relaxed.

\begin{example}
    Let $(K_1,\ldots,K_n)$ be a graph-directed attractor in $\R$ associated with $G=G(\V,\E)$ and $\Phi=\{\vp_e: e\in\E\}$, where
    \begin{enumerate}
        \item for each $e\in\E$, $\vp_e$ is a contracting map on $\R$ of the form $\vp_e(x)=r_ex+t_e$, where $0<r_e<1$ and $0\leq t_e\leq 1-r_e$;
        \item for every pair of $e,e'\in\E$, the two intervals $\vp_e([0,1])$ and $\vp_{e'}([0,1])$ are either the same or adjacent.
    \end{enumerate}
    In this setting, if some pair of level-$k$ intervals are not disjoint then they are either the same or adjacent. So one can follow the arguments in this paper to see whether $K_i\cap K_j\neq\varnothing$ for $1\leq i,j\leq n$.
\end{example}

Putting some restrictions on the symmetries, we are also able to determine the connectedness of self-affine sponge-like sets.

\begin{example}
    Let $n,m\geq 2$ be integers and let $0<a_1,\ldots,a_n<1$, $0<b_1,\ldots,b_m<1$ be such that $\sum_{p=1}^n a_p=\sum_{q=1}^m b_q=1$. Let $\D\subset\{0,1,\ldots,n-1\}\times\{0,1,\ldots,m-1\}$ be a non-empty digit set. For each $(i,j)\in\D$, set
    \[
        \psi_{(i,j)}: \Big(\begin{array}{c} x \\ y \end{array}\Big) \mapsto O_{(i,j)}\Big(\begin{array}{cc} a_{i+1} & 0 \\ 0 & b_{j+1} \end{array}\Big)\Big(\begin{array}{c} x \\ y \end{array}\Big) + \Big(\begin{array}{c} \sum_{p=1}^i a_p \\ \sum_{q=1}^j b_q \end{array}\Big),
    \]
    where $O_{(i,j)}$ is an element of the collection of rotations of $0^\circ, 180^\circ$ around the point $(1/2,1/2)$ and flips along the lines $x=1/2$ or $y=1/2$. When $O_{(i,j)}=\id$ for all $(i,j)\in\D$, the attractor associated with the IFS $\{\psi_{(i,j)}:(i,j)\in\D\}$ is called a Bara\'nski carpet (see~\cite{Bar07}); if we further have $a_p\equiv n^{-1}$ and $b_q\equiv m^{-1}$ then the attractor is called a Bedford-McMullen carpet (see~\cite{Fra20}). In this case, it is easy to see that the intersection of every pair of level-$k$ rectangles is their largest common face and our arguments work in this setting.
\end{example}

%------------------------------
\bigskip

\noindent{\bf Acknowledgements.}
The research of Ruan is partially supported by NSFC grant 12371089, ZJNSF grant LY22A010023 and the Fundamental Research Funds for the Central Universities of China grant 2021FZZX001-01 and 226-2024-00136. The research of Xiao is partially supported by the General Research Funds (CUHK14301017, CUHK14301218) from the Hong Kong Research Grant Council. We are grateful to Professor Kenneth Falconer for suggesting the connectedness question of sponge-like sets and his helpful comments. We also thank the anonymous referee reading the paper very carefully, and for providing a large number of valuable suggestions.

\small
\bibliographystyle{amsplain}

\end{document}